\RequirePackage[l2tabu, orthodox]{nag}
\documentclass[11pt]{amsart}
\usepackage[letterpaper]{geometry}
\usepackage{graphicx, color, booktabs, tabularx, hyphenat, parskip, bbm, bm}
\usepackage{enumerate, mathtools, amssymb} 
\usepackage[font=small, format=plain, labelfont=bf, up, textfont=it, up]{caption}
\usepackage[normalem]{ulem}
\usepackage[perpage]{footmisc}
\usepackage{float}   
\usepackage{microtype}
\usepackage[hyperfootnotes=false, colorlinks, allcolors=blue]{hyperref}
\usepackage{cleveref}
\usepackage[backend=bibtex,style=numeric]{biblatex}
\bibliography{zeta}

\numberwithin{equation}{section}

\DefineFNsymbols*{lamportnostar}[math]{\dagger\ddagger\S\P\|{\dagger\dagger}{\ddagger\ddagger}}
\setfnsymbol{lamportnostar}
\def\@fnsymbol#1{\ensuremath{\ifcase#1\or \dagger\or \ddagger\or
   \mathsection\or \mathparagraph\or \|\or \dagger\dagger\or 
   \ddagger\ddagger \else\@ctrerr\fi}}

\newtheorem{thm}{Theorem}[section]
\newtheorem{lemma}[thm]{Lemma}

\newtheorem{cor}[thm]{Corollary}

\theoremstyle{definition}
\newtheorem{rem}[thm]{Remark}

\newtheoremstyle{RHP}  		
  {3pt}					
  {3pt}					
  {\itshape}				
  {}						
  {\bfseries}				
  {}						
  {.5em}					
  {Riemann-Hilbert Problem #2 #3} 

\theoremstyle{RHP}

\makeatletter
\let\save@mathaccent\mathaccent
\newcommand*\if@single[3]{%
  \setbox0\hbox{${\mathaccent"0362{#1}}^H$}%
  \setbox2\hbox{${\mathaccent"0362{\kern0pt#1}}^H$}%
  \ifdim\ht0=\ht2 #3\else #2\fi
  }
\newcommand*\rel@kern[1]{\kern#1\dimexpr\macc@kerna}
\newcommand*\widebar[1]{\@ifnextchar^{{\wide@bar{#1}{0}}}{\wide@bar{#1}{1}}}
\newcommand*\wide@bar[2]{\if@single{#1}{\wide@bar@{#1}{#2}{1}}{\wide@bar@{#1}{#2}{2}}}
\newcommand*\wide@bar@[3]{%
  \begingroup
  \def\mathaccent##1##2{%
    \let\mathaccent\save@mathaccent
    \if#32 \let\macc@nucleus\first@char \fi
    \setbox\z@\hbox{$\macc@style{\macc@nucleus}_{}$}%
    \setbox\tw@\hbox{$\macc@style{\macc@nucleus}{}_{}$}%
    \dimen@\wd\tw@
    \advance\dimen@-\wd\z@
    \divide\dimen@ 3
    \@tempdima\wd\tw@
    \advance\@tempdima-\scriptspace
    \divide\@tempdima 10
    \advance\dimen@-\@tempdima
    \ifdim\dimen@>\z@ \dimen@0pt\fi
    \rel@kern{0.6}\kern-\dimen@
    \if#31
      \overline{\rel@kern{-0.6}\kern\dimen@\macc@nucleus\rel@kern{0.4}\kern\dimen@}%
      \advance\dimen@0.4\dimexpr\macc@kerna
      \let\final@kern#2%
      \ifdim\dimen@<\z@ \let\final@kern1\fi
      \if\final@kern1 \kern-\dimen@\fi
    \else
      \overline{\rel@kern{-0.6}\kern\dimen@#1}%
    \fi
  }%
  \macc@depth\@ne
  \let\math@bgroup\@empty \let\math@egroup\macc@set@skewchar
  \mathsurround\z@ \frozen@everymath{\mathgroup\macc@group\relax}%
  \macc@set@skewchar\relax
  \let\mathaccentV\macc@nested@a
  \if#31
    \macc@nested@a\relax111{#1}%
  \else
    \def\gobble@till@marker##1\endmarker{}%
    \futurelet\first@char\gobble@till@marker#1\endmarker
    \ifcat\noexpand\first@char A\else
      \def\first@char{}%
    \fi
    \macc@nested@a\relax111{\first@char}%
  \fi
  \endgroup
}
\makeatother

\newcommand{\R}{\mathbb{R}}
\newcommand{\C}{\mathbb{C}}
\newcommand{\N}{\mathbb{N}}

\newcommand{\ie}{{i.e.}}

\newcommand{\cf}{cf. }

\newcommand{\U}{\mathcal{U}}

\newcommand{\nin}{\not\in}
\newcommand{\lp}{\left(}
\newcommand{\rp}{\right)}
\newcommand{\lb}{\left[}
\newcommand{\rb}{\right]}
\DeclarePairedDelimiter\ceil{\lceil}{\rceil}
\DeclarePairedDelimiter\floor{\lfloor}{\rfloor}

\renewcommand{\Re}{\operatorname{Re}}
\renewcommand{\Im}{\operatorname{Im}}

\newcommand{\bigo}[2][]{\mathcal{O}_{#1} \left( #2 \right)}

\newcommand{\dd}{\mathrm{d}}

\newcommand{\eps}{\epsilon}

\newcommand{\g}{\gamma}

\newcommand{\zerofree}{\mathcal{F}_\lambda}
\newcommand{\zerofreeone}{\mathcal{F}_{1,\lambda} }
\newcommand{\K}{\mathrm{K}}
\newcommand{\y}{\mathcal{Y}}
\newcommand{\Acal}{\mathcal{A}}
\newcommand{\Ecal}{\mathcal{E}}
\newcommand{\ball}{B}
\newcommand{\contour}{\gamma}

\newcommand{\scal}{ \mathcal{S} }
\newcommand{\dscal}{ \partial \mathcal{S} }
\newcommand{\phase}{ \phi_{\lambda} }
\newcommand{\Sz}{\Omega}
\newcommand{\nSz}{\mho}
\newcommand{\up}{\upsilon}

\DeclareMathOperator{\erfc}{erfc}

\DeclareMathOperator{\W}{W}
\DeclareMathOperator{\dist}{dist}

\title[Zeros of $\zeta$ polynomials]{Dynamic behavior of the roots of the Taylor polynomials of the Riemann xi function with growing degree}
\author{Robert Jenkins}
\address{Department of Mathematics  University of Arizona, Tucson, Arizona 85721}
\email{\texttt{rjenkins@math.arizona.edu}}

\author{Ken D. T.-R. McLaughlin}
\address{Department of Mathematics  Colorado State University, Fort Collins, CO 80526}
\email{\texttt{kenmcl@rams.colostate.edu}}

\date{\today}

\begin{document}

\begin{abstract}
We establish a uniform approximation result for the Taylor polynomials of the xi function of Riemann which is valid in the entire complex plane as the degree grows. In particular, we identify a domain growing with the degree of the polynomials on which they converge to Riemann's xi function. 
Using this approximation we obtain an estimate of the number of ``spurious zeros" of the Taylor polynomial which are outside of the critical strip, which leads to a Riemann - von Mangoldt type of formula for the number of zeros of the Taylor polynomials within the critical strip.  Super-exponential convergence of Hurwitz zeros of the Taylor polynomials to bounded zeros of the xi function are established along the way, and finally we explain how our approximation techniques can be extended to a collection of analytic L-functions.
\end{abstract}

\maketitle

\section{Introduction}
Consider Riemann's $\xi$-function defined by
\begin{equation}\label{xi}
	\xi(z) = \frac{1}{2} \pi^{-z/2}\, \Gamma \lp \frac{z}{2} \rp  z (z-1)  \zeta(z)
\end{equation}
where $\zeta(z)$ is the Riemann $\zeta$-function. 
The pre-factors of the $\zeta$-function in the above definition absorb the poles and trivial zeros of the $\zeta$-function so that $\xi$ is an entire function whose only zeros are the nontrivial zeros of $\zeta(z)$, \ie\ those lying in the critical strip $0 < \Re\, z < 1$. As a consequence, the functional equation for the $\xi$-function is much simplified
\begin{equation}\label{relation}
	\xi(z) = \xi(1-z).
\end{equation}
The infamous Riemann Hypothesis is equivalent to the statement that the only zeros of $\xi(z)$ lie on the critical line $\Re z = 1/2$. There is a vast body of literature concerning the properties of the $\zeta$-function and the Riemann Hypothesis, and we cannot do any justice to summarizing those works here. We refer the reader to the classical works \cite{Edwards,Titchmarsh} at the tip of that iceberg.

In Riemann's 1859 paper \cite{Riemann} he considered the quantity
\[
	N(T) = 
	\left\{ 
	z \in \C \mid \zeta(z) = 0, \quad 
	\Re(z) \in (0,1), \ \Im(z) \in (0,T] 
	\right\},
\]
and proposed that $N(T) \approx T/(2\pi) \log(T/(2\pi))-T/(2\pi)$. This was subsequently proved by von Mangoldt with an explicit bound on the remainder:
\begin{equation}\label{Riemann-Mangoldt}
	N(T) = \frac{T}{2\pi} \log \frac{T}{2\pi}  - \frac{T}{2\pi} + \bigo{ \log T }.
\end{equation}

In this paper we study the zeros of the Taylor polynomial approximations of Riemann's $\xi$-function. 
We establish a version of the Riemann-von Mangoldt formula for these zeros, by using a new uniform asymptotic description of the Taylor polynomials when the degree is large. 
The techniques used here are general and can be applied to a broad class of functions. 
In the last section of this paper we extend the analysis of Taylor polynomials to a larger collection of analytic L-functions. 

Studying the zeros of Taylor approximates to given functions goes back at least to the 1920s, and probably earlier. 
In \cite{Szego}, Szeg\H{o} considered the distribution of zeros of 
$p_n(z) = \sum_{k=0}^n z^k/k!$, the partial sums of the exponential series. 
He showed that the zeros of the rescaled function $p_n(nz)$ converge as $n \to \infty$ to a curve $D_\infty$, now called the Szeg\H{o} curve, which is a branch of the level curve $\{ z \,:\, \left| z e^{1-z} \right| = 1 \}$ and computed the asymptotic distribution of zeros along the Szeg\H{o} curve. 
Subsequent work in this direction \cite{ACV,Buckholtz,NR} has provided detailed results bounding the distance of the zeros of $p_n(nz)$ from $D_\infty$.
Similar results on the zeros of the partial sums of $\cos,\ \sin$, and other exponential functions have been derived in \cite{Bleher,CVW}.  An extension to partial sums of analytic functions defined by exponential integrals appears in \cite{Vargas}, which also contains some further historical discussion and references.

The starting point for our analysis of the Taylor polynomials of the $\xi$-function is the recent work of \cite{KKM} in which the authors utilize basic facts of complex analysis to represent the partial sums, $p_n$, of the exponential series as Cauchy integrals over certain contours in the complex plane. 
Steepest descent analysis of the (rescaled) Taylor polynomials and properties of Cauchy integrals lead to a uniform asymptotic description of the polynomials as $n \to \infty$ in the entire complex plane. 
In doing so, \cite{KKM} re-derives many of Szeg\H{o}'s classic results on the zeros of the partial sums. Additionally, the method naturally accommodates the presence of critical points in the asymptotic analysis which complicate the approximation theory in the more classical works mentioned previously.  

Recall the Cauchy integral representation of the $n^{th}$ Taylor polynomial approximating a given function $F(z)$  (which we assume to be entire to avoid fretting about domain issues):
\begin{align}
\label{ChangeThis}
T_{n-1}(F;z) = F(z) \left[ \chi_{\mathcal{S}}(z) - \frac{z^{n}}{F(z)} \oint_{\partial \mathcal{S}} \frac{F(s)}{s^{n}} \frac{ds}{s-z}  \right]\ ,
\end{align}
where $\mathcal{S}$ is taken to be a simply connected open set whose boundary $\partial \mathcal{S}$ is either a finite union of smooth arcs forming a simple closed (obviously rectifiable) curve, or a reasonable extension (which will be described as needed below), and $\chi_{\mathcal{S}}(z)$ is the characteristic function of $\mathcal{S}$.  Basic results concerning Taylor approximation are obtained from this representation by taking $\mathcal{S}$ to be a disc of fixed size and then estimating the $n$-dependence of the last term on the right hand side of (\ref{ChangeThis}).  More interestingly, the integral's dependence on $n$ can be estimated precisely, using the steepest descent method for integrals, provided the function $F$ is so nice as to permit the application of the steepest descent method.  

A portion of this paper is dedicated to showing very explicitly that this is so for the function $\xi$ defined above.  However, it is useful to describe the general conditions, as cryptic as they might appear to be:  one requires that for $n$ sufficiently large there should be a number of ``stationary phase points'', and that the original contour of integration can be deformed to a contour of controllable arc length which passes through one or more of these stationary phase points while otherwise remaining in regions where the integrand is exponentially smaller than its behavior near one (or more) of these critical points.  
\begin{figure}[htb]
\centering
	\includegraphics[width=.6\textwidth]{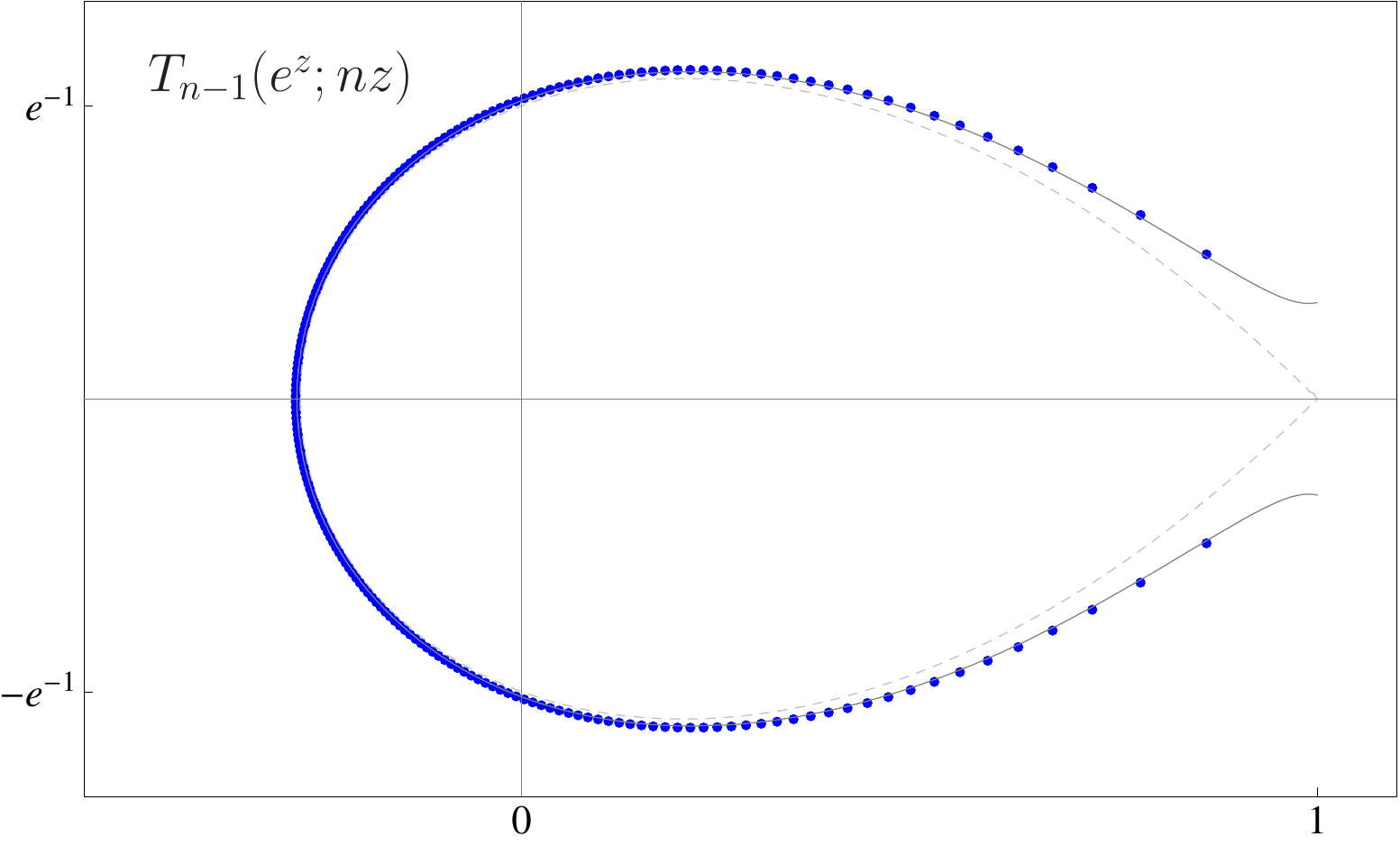}
\caption{
Each dot represents a zero of the (rescaled) Taylor polynomial $T_{n-1}(e^z; nz)$ of degree $n-1=200$. As $n\to \infty$, these zeros accumulate along the Szeg\H{o} curve $D_\infty$ (dashed line); for finite $n$, an improved Szeg\H{o} curve $D_n$ (solid) line better approximates the zeros. 
\label{fig:exp}
}
\end{figure}
The simple case of $F(z) = e^{z}$ is useful to clarify the above discussion (see \cite{KKM} for more information, including a brief discussion of the various contributions to this example).  
Evaluating \eqref{ChangeThis} by steepest descent methods, it is convenient to introduce a rescaling map $z \mapsto \lambda(n) z$ which renormalizes the stationary phase points, which typically grow with $n$, to remain $\bigo{1}$ as $n \to \infty$. In the case $F(z) = e^{z}$, there is a single stationary phase point $z_0 =n := \lambda(n)$ and \eqref{ChangeThis} becomes\footnote{The behavior for $z$ near $1$ is more delicate, since the steepest descent method must be modified to accommodate a pole impinging upon a stationary phase point} for any $\delta>0$,  
\begin{align}\label{TaylorApproxExp}
T_{n-1}(e^{z}; n z) = 
	e^{n z}  \left[ \chi_{\mathcal{S}}(n z) - 
	 \frac{ \left( z e^{1-z} \right)^{n}}{ \sqrt{2 \pi } \sqrt{ n } }\frac{1}{1-z} \left( 1 + \bigo{\frac{1}{n} } \right) \right], \qquad |z-1| >\delta.
\end{align}
Formula \eqref{TaylorApproxExp} demonstrates that the Taylor polynomials approximate $e^z$ on sets that grow with $n$.  We can characterize the largest such set, $\Omega(e^z)$, as the closure of the connected component of $|z e^{1-z}|< 1$ containing $z=0$:
\begin{align}
\Omega\left( e^{z} \right) = 
\left\{ z \,: \,  |z e^{1-z}| < 1 \text{ and } |\Re z| < 1
\right\} \ .
\end{align}
The boundary $D_\infty = \partial \Omega( e^z )$ is the Szeg\H{o} curve mentioned previously. 
For $z$ away from the Szeg\H{o} curve, the asymptotic formula \eqref{TaylorApproxExp} clearly cannot vanish.
Szeg\H{o} showed that: 
1) every accumulation point of the zeros $\{z_{k,n} \}_{k=1}^n$ of $T_{n-1}(e^z; nz)$ must lie on $D_\infty$; 
2) Every point on $D_\infty$ is an accumulation point of $\{z_{k,n} \}_{k=1}^n$. 
It was later shown, \cite{Buckholtz}, that $\dist \lp z_{k,n};\ D_\infty \rp = \bigo{\frac{\log n}{n} }$ for each zero $z_{k,n}$ of $T_{n-1}(e^z; n z)$ which is uniformly bounded away from the stationary point at $z=1$ (for $z_{k,n}$ near $1$ the rate of convergence to $D_\infty$ slows to $\bigo{n^{-1/2}}$). 
It's also possible to improve on the Szeg\H{o} curve; one can consider the curve
\begin{equation}\label{exp Szego1}
	D_n^{(1)} = \left\{ z \,: \,  
		\frac{|z e^{1-z}|^n}{\sqrt{2\pi n}|1-z|} = 1 \text{ and } |\Re z| < 1
	\right\};
\end{equation}
it was shown in \cite{CVW} that for any $\delta>0$, $\dist(z_{k,n}, D_n^{(1)}) = \bigo{n^{-2}}$ for each $z_{k,n}$ such that $|z_{k,n}-1|> \delta$. The curve $D_n^{(1)}$ is only the first in a countable family of improved Szeg\H{o} curves $D_n^{(j)}$; the further improved Szeg\H{o} curves result from keeping $j$ terms from the complete asymptotic series which in \eqref{TaylorApproxExp} is represented simply by $\lp 1+\bigo{n^{-1}}\rp$.
In Figure~\ref{fig:exp} we plot the Szeg\H{o} curve and its (first) improvement for $e^z$ along with the roots of $T_{n-1}(e^z; nz)$ for $n=201$.  The plot was produced using the software package Mathematica \cite{Mathemagica}. 
\begin{figure}[htb]
\centering
	\includegraphics[width=.7\textwidth]{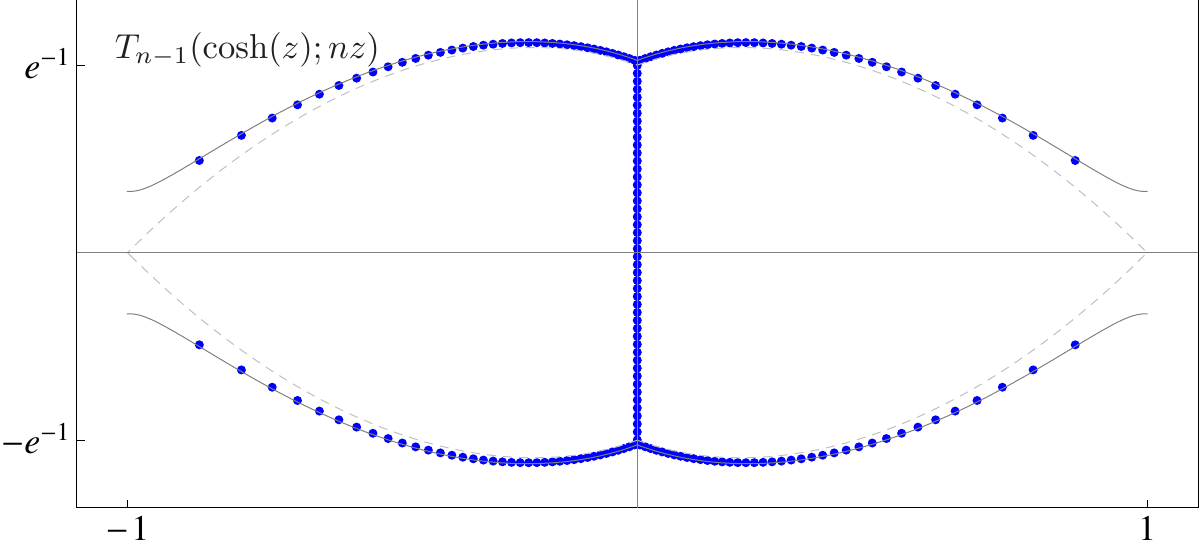}
\caption{
Each dot represents a zero of the (rescaled) Taylor polynomial $T_{n}(\cosh(z); (n+1)z)$ of degree $n=200$. The Szeg\H{o} curve (dashed line); and an improved Szeg\H{o} curve (solid line) are also given. Here, the zeros in the imaginary interval $[-ie^{-1}, ie^{-1}]$ are the Hurwitz zeros of $T_n(\cosh(z);z)$.
\label{fig:cosh}
}
\end{figure}

The situation for functions which have zeros is somewhat modified. 
Suppose that $s$ is a root of order $k$ of a function $f$ analytic at $s$. 
Then given any sufficiently small neighborhood $\mathcal{N}$ of the root $s$, the Taylor polynomials $T_n(f;z)$ converge (uniformly) to $f$ in $\mathcal{N}$ and so by Hurwitz's theorem (\cf \cite{Conway}) $T_{n}(f;z)$ will have exactly $k$ zeros in $n$ for all $n \geq n_0(\mathcal{N})$. 
This imposes a natural dichotomy on the zeros of the Taylor polynomials: those which converge to the zeros of $f$ we label, `Hurwitz zeros'; those which do not converge to zeros of $f$ we label `spurious zeros', and these accumulate on the analogue of the Szeg\H{o} curve for the function $f$. To illustrate this dichotomy see Figures~\ref{fig:cosh} and \ref{fig:zeta} where the zeros of rescaled Taylor polynomials of $\cosh(z)$ and $\xi(z+1/2)$ are given together with their Szeg\H{o} curves. 

In both Figure~\ref{fig:cosh} and Figure~\ref{fig:zeta} the zeros of the functions $\cosh((n+1)z)$ and $\xi(\lambda(n) z+1/2)$ do not appear, because they agree with the computed zeros of the Taylor polynomials to well beyond the plotting resolution.
In Table~\ref{table:cosh_zeros} the 24 roots of $T_{n}(\cosh(z); (n+1) z)$ with $n=200$ which lie on the imaginary axis in Figure~\ref{fig:cosh} are compared to the first 24 zeros of $\cosh((n+1)z)$. 
The convergence rate is striking.  These numerical calculations required very high precision calculations using \cite{Mathemagica}.
In Section~\ref{sec:true zeros} below, we will show that the rate at which any fixed Hurwitz zeros converges to a fixed root of the function $\xi$ is super-exponential. We believe that this is true for a large class of entire functions $f$, of which, as Table~\ref{table:cosh_zeros} suggests, $\cosh$ is certainly a member. 
\begin{table}
\renewcommand{\arraystretch}{1.1}
\centering
\begin{tabular}[t]{ @{}l@{\qquad}l@{\hskip 0.1\linewidth}l@{\qquad}l@{\hskip 0.1\linewidth}l@{\qquad}l@{}}
\toprule
$k$ & $|\tfrac{(2k-1)\pi}{2*201} - z_{k,n}|$ &
$k$ & $|\tfrac{(2k-1)\pi}{2*201} - z_{k,n}|$ &
$k$ & $|\tfrac{(2k-1)\pi}{2*201} - z_{k,n}|$ \\ 
\midrule
1 & $6.4203*10^{-343}$ &  9 & $2.2431*10^{-94}$ & 17 & $3.3118*10^{-36}$ \\
2 & $1.5341*10^{-246}$ & 10 & $1.2781*10^{-84}$ & 18 & $4.7719*10^{-31}$ \\
3 & $9.9742*10^{-202}$ & 11 & $7.6667*10^{-76}$ & 19 & $3.5500*10^{-26}$ \\
4 & $3.2819*10^{-172}$ & 12 & $7.2966*10^{-68}$ & 20 & $1.4618*10^{-21}$ \\
5 & $3.6516*10^{-150}$ & 13 & $1.4982*10^{-60}$ & 21 & $3.5351*10^{-17}$ \\
6 & $1.4648*10^{-132}$ & 14 & $8.4059*10^{-54}$ & 22 & $5.2813*10^{-13}$ \\
7 & $6.6037*10^{-118}$ & 15 & $1.5514*10^{-47}$ & 23 & $5.0926*10^{-9}$  \\
8 & $2.3563*10^{-105}$ & 16 & $1.0925*10^{-41}$ & 24 & $3.2346*10^{-5}$  \\	
\bottomrule  
\end{tabular}
\caption{Differences between the 24 numerical calculated Hurwitz zeros $z_{k,n}$ of $T_{200}(\cosh(z); 201 z)$ on the critical line depicted in Figure~\ref{fig:cosh}, and the first 24 zeros of $\cosh( 201 z)$. Numerical calculations were done with 400 digits of working precision \cite{Mathemagica}.  
\label{table:cosh_zeros}
} 
\end{table}

\subsection{Taylor polynomials of \texorpdfstring{$\xi$}{xi}}

In the remainder of the paper we will be interested in the Taylor (Maclaurin) polynomials of the function 
\begin{equation}\label{f def}
	f(z) = \xi(1/2 + z).
\end{equation}
The function $f$ is entire and possesses the symmetries $f(z)^* = f(z^*)$ and $f(-z) = f(z)$, the later of which follows from \eqref{relation}.
The Taylor polynomials $T_n$ inherit the symmetries of $f$; $T_n(f;z^*)^*=T_n(f;z) = T_{n}(f;-z)$, so that for any $n \in \N$, 
$i)$ $T_{2n+1}(f;z) = T_{2n}(f;z)$;
and $ii)$ zeros of $T_n$, excepting purely real or imaginary roots, come in quartets.
In what follows we will omit the dependence of the Taylor polynomials 
upon $f$ and write simply $T_{2n}(z)$ for $T_{2n}(f;z)$.

The exponential decay of $|\Gamma(z)|$ along vertical lines---the other factors in \eqref{xi} being polynomially bounded in $\Im(z)$---allows us to deform the set $\mathcal{S}$ in \eqref{ChangeThis} to an infinite vertical strip.
For any number $\lambda >0$, let 
\begin{equation}\label{set}
	\scal_\lambda = \left\{ z \in \C :\, | \Re z | < \lambda \right\}.
\end{equation}
Anticipating the introduction of a scaling parameter $\lambda = \lambda(n)$, 
and letting $\chi = \chi_{\scal_1}$ be the characteristic function of $\scal_1$, we have 
\begin{equation}\label{represent}
	\begin{aligned}
    	T_{2n-2}(\lambda z) 
    	= f(\lambda z) \lb \chi(z) - \frac{e^{n \phase(z) }}{\sqrt{n}}  h(z) \rb
	\end{aligned}
\end{equation}
where we have defined
\begin{gather}
	\label{phase}
		e^{n \phase(z)} := \frac {z^{2n} f(\lambda) }{f(\lambda z) }, \\
	\label{H}
		h(z) := \frac{ \sqrt{n}}{ 2\pi i } \int_{ \dscal_1 } 
		e^{-n \phase(s)} \frac{ds}{s-z}.
\end{gather}

\section{Preliminaries}

The methods of Korobov and Vinogradov produce the following zero free region (c.f.~\cite[\S 6.19]{Titchmarsh}) of $\zeta$ extending inside the critical strip: for any choice of $A>0$, $\zeta(s)$ has no zeros for $s = \sigma + it$,  $\sigma, t \in \R$ with $|t|$ large and $\sigma > 1 - \frac{A}{ (\log t)^{2/3} (\log \log t)^{1/3}}$ and we have the bounds
\begin{equation}\label{zeta bounds}
	\begin{gathered}
	|\zeta(s)|  = \bigo{ (\log t)^{2/3} ( 1 + |t|^{100(1-\sigma)^{3/2}}) },  
	\qquad 1/2 \leq \sigma \leq 1, \\
	\frac{\zeta'(s)}{\zeta(s)} = \bigo{(\log t)^{2/3} (\log\log t)^{1/3} }, \qquad 
	\frac{1}{\zeta(s)} = \bigo{(\log t)^{2/3} (\log\log t)^{1/3} }. \\
	\end{gathered}
\end{equation}
the best bounds of this type are those of Ford \cite{Ford}.

It follows that our rescaled function $f(\lambda z)$ is zero free in the domain 
\begin{equation}\label{zero-free}
	\zerofree = \left \{ z = x+ i y \in \C^+ \,:\,  x \geq \frac{1}{2\lambda} - \frac{A}{\lambda  (\log \lambda y )^{2/3} (\log\log \lambda y)^{1/3}} \right\}.
\end{equation}

Outside the critical strip we have the more elementary bound from \cite{Ford}
\begin{lemma}\label{lem:dlog zeta}
Let $s= \sigma + i t$ with $\sigma,t \in \R$ and $\sigma > 1$, then
\[
	\left| \frac{ \zeta'(s)}{\zeta(s)} \right| \leq \frac{1}{\sigma-1}.
\]
\end{lemma}
\begin{proof}
For $\sigma>1$ we have $|\zeta'(s)/\zeta(s)| \leq -\zeta'(\sigma)/\zeta(\sigma)$ and
\begin{gather*}
	-\zeta'(\sigma) = \sum_{m=2}^\infty \frac{\log m}{m^\sigma} 
	= \sum_{m=2}^\infty \lb \sum_{n=1}^{m-1} \log \lp \frac {n+1}{n} \rp \rb m^{-\sigma}
	= \sum_{n=1}^\infty \lb \sum_{m \geq n+1} m^{-\sigma} \rb \log \lp \frac {n+1}{n} \rp.
\intertext{The result follows from bounding the interior sum by the integral $\int_n^\infty m^{-\sigma} dm$ and recalling that for $x>0$, $\log(1+x) < x$:}  
	-\zeta'(\sigma) \leq 
	\sum_{n=1}^\infty \lp \frac{ n^{1-\sigma}}{\sigma-1} \rp \frac {1}{n} 
	= \frac{\zeta(\sigma)}{\sigma-1}. \qedhere
\end{gather*}

\end{proof}

The above bounds on the logarithmic derivative, both near the strips edge and outside it, give a bound on the argument of $\zeta(s)$ at the edge of the critical strip. 
\begin{lemma}\label{lem:arg bound}
There exist $t_0>0$ such that for all $t> t_0$ we have
\[
	\arg \zeta(1 + it) \leq  \tfrac{2}{3}\log\!\log t + \bigo{\log\!\log\!\log t}.
\]
\end{lemma}
\begin{proof}
Since $\zeta(2)>0$ and $\Re \zeta(2 + i \tau) \geq 1 - \sum_{n=1}^\infty n^{-2} > 0$ for all $\tau \geq 0$, $\Re \zeta$ is strictly positive on the vertical line from $s=2$ to $s=2+it$. 
It follows that $|\arg \zeta(2+it)| \leq \pi$. 
Using \eqref{zeta bounds} and Lemma~\ref{lem:dlog zeta} for all sufficiently large $t$ there exist a constant $A>0$, such that for any $q \in (0,1)$ we have
\begin{multline*}
	\left| \arg \zeta(1+it)  - \arg \zeta(2+it) \right| 
	\leq \int_{1+q}^2 \frac{d \sigma}{\sigma-1}  + A q (\log t)^{2/3} (\log\log t)^{1/3} \\
	= \log \frac{1}{q} + A q (\log t)^{2/3} (\log\log t)^{1/3}.
\end{multline*}
The minimizer of this last expression, as a function of $q$, is $q_0 = A^{-1}(\log t)^{-2/3} (\log\log t)^{-1/3}$.
Computing the minimum completes the proof.
\end{proof}

\subsection{The phase \texorpdfstring{$\phase(z)$}{} }
The phase, implicitly defined by \eqref{phase},
\begin{equation}\label{phase def}
	\phase(z) = 2 \log z + \frac{1}{n} \log \frac{ f(\lambda)}{f(\lambda z)},
\end{equation}
is analytic in any region in which $f(\lambda z) = \xi(1/2 + \lambda z)$ is zero free.  
In particular $\phase$ is well defined along the contour of integration $|\Re z| = 1$.
Moreover, the choice of branch can be chosen such that $\phase(z)$ is positive real for $z \in (1,\infty)$ and satisfies the symmetry $\phase(z) = \phase(-z)$. 

The following formula for $\phase$ is well suited for a large $\lambda$ expansion.
For any fixed $c>0$, if $|z|>c$ and $\lambda \gg 1$ we have
\begin{gather}\label{phase expand 0}
\begin{multlined}[.9\linewidth]
	\phase(z) = 2 \log z 
	+ \lp \frac{\lambda}{2n} \log \frac{\lambda}{2\pi}\rp (1-z) 
	- \frac{\lambda}{2n} \lb 1 - z + z \log z \rb  \\
	- \frac{1}{n} \log \zeta \lp \lambda z + \frac{1}{2} \rp
	+ \frac{1}{n} \, r(z;\lambda)
\end{multlined}	
\intertext{where the remainder $r(z; \lambda)$ is given by}
\begin{multlined}[.9\linewidth]
	r(z; \lambda) =
	\log \frac{\Gamma(\lambda/2+1/4)}{\Gamma(\lambda z/2+1/4)} 
	- \lp \frac{\lambda}{2} \log \frac{\lambda}{2} \rp (1-z) 
	+ \frac{\lambda}{2} \lb 1-z + z \log z \rb \\
	+ \log \lb  \frac{(\lambda^2 - 1/4) \zeta(\lambda+1/2)}
	{(\lambda^2 z^2 - 1/4) } 
	\rb. 
\end{multlined}
\end{gather}
This remainder term is bounded provided that $z$ stays away from its obvious singularities. More precisely, let $c>0$ be fixed, then using Stirling's expansion of $\log \Gamma(s)$, one may verify that 
\begin{equation}\label{phase remainder}
	r(z;\lambda) = \bigo{1}. \qquad  \Re z \geq \frac{1}{2\lambda} \text {  and } |z -  \frac{1}{2\lambda}| > c.
\end{equation}
The explicit $\zeta$ term in \eqref{phase expand 0} becomes meaningful only near the critical strip; elsewhere, it is comparable to the remainder $r$.
One can similarly compute the $z$-derivative of the phase:
\begin{align}\label{phase' expand}
	\partial_z \phase(z) = 
	\frac{2}{z} - \frac{\lambda}{2n} \log \frac{\lambda}{2\pi} 
	- \frac{\lambda}{2n} \log z 
	- \frac{\lambda}{n} \frac{\zeta'(\lambda z + 1/2)}{\zeta(\lambda z + 1/2)} 
	+ \frac{1}{n} \partial_z r(z;\lambda) .
\end{align}

The representation \eqref{represent} places the essential $n$-dependence of the Taylor polynomials in the phase $\phase$ defined by \eqref{phase} which appears in the exponential term of the integral \eqref{H}. 
As the following lemma shows, for large $n$ the phase has two stationary points outside the critical strip, and these points' magnitudes  increase with $n$. We choose the scaling parameter $\lambda = \lambda(n)$ according to Lemma~\ref{lem:crit pts} below precisely so that these stationary points lie at $z = \pm 1$ in the rescaled plane\footnote{By symmetry the stationary points must be opposites.}. This completes the definition of $T_{2n-2}(\lambda z)$ so that the representation \eqref{represent} is now well defined.

\begin{lemma}\label{lem:crit pts}
For all sufficiently large $n$ there is a unique choice of $\lambda = \lambda(n)$, with $\lambda > 1/2$ (\ie\ right of the shifted critical strip) satisfying 
$
	\partial_z \phase(z) \Big|_{z=1} = 
	2 - (\lambda/n) \partial_\lambda \log f(\lambda) = 0.
$
This choice of $\lambda$ satisfies the relation
\begin{equation}\label{lambda-func-asymp}
	2 - \frac{\lambda}{2n} \log \lp \frac{\lambda}{2\pi} \rp 
	= \bigo{\frac{1}{n}},
\end{equation}
and asymptotically 
\[
	\lambda = \lambda(n) 
	= \frac{4n}{\W \lp 2n / \pi \rp }\lb 1 + \bigo{n^{-1}} \rb.  	
\]
Here $\W(z)$ is the branch of the inverse function to $\W e^{\W} = z$ which is real and increasing for $z \in (-e^{-1},\infty)$ sometimes called the Lambert-$\W$ function\footnote{
For more information on $\W(z)$ see \S 4.13 of \cite{DLMF}
}. 

Moreover, for this choice of $\lambda$ the critical point at $z=1$ is simple and
\begin{equation}\label{phase constant}
	\phase''(1) = -2 + \bigo{ \frac{1}{\log n} }. 
\end{equation}
\end{lemma}

\begin{proof}
As $f$ is entire, $\partial_\lambda \log f(\lambda)$ is bounded for any finite $\lambda$ outside the open critical strip as $f(\lambda)$ is zero free in this region. 
It follows that any root $\lambda(n)$ of $2 - (\lambda/n) \partial_\lambda \log f(\lambda)$ outside the strip must grow without bound as $n \to \infty$. 

Let $\eps = 1/n$,
\[
		G(\epsilon,\lambda) = 2 - \epsilon \lambda \partial_\lambda \log f(\lambda), 
		\qquad 
		\lambda(\epsilon, \nu) = \frac{4}{\epsilon W(2/(\epsilon \pi))} \lb 1 + \epsilon \nu \rb,
\]
and let $\widehat{G}(\eps, \nu) = \epsilon^{-1} G(\epsilon,\lambda(\epsilon,\nu))$. 
As $\phase'(1) = 0$ is equivalent to $G(\eps, \lambda) = 0$, the theorem is proved if we can show that $\widehat{G}(\eps,\nu) = 0$ implicitly defines a unique function $\nu(\eps)$ which is bounded for $\eps$ near $0$. 
Using \eqref{phase' expand} we have
\begin{gather*}
	G(\epsilon,\lambda) 
	= 2 - \frac{\epsilon \lambda}{2} \log\lp \frac{\lambda}{2\pi} \rp 
	- \eps\, R(\lambda)
\shortintertext{where $R$ is given by}
	R(\lambda) := \frac{\lambda}{2} 
	\lb 
		\psi \lp \frac{\lambda}{2} + \frac{1}{4} \rp 
		- \log \frac{\lambda}{2} + \frac{4}{\lambda} \lp 1 + \frac{1}{4\lambda^2 -1} \rp
		+ \frac{2 \zeta'(\lambda+1/2)}{\zeta(\lambda+1/2)} 
	\rb	
\end{gather*}
Here $\psi$ denotes the digamma function, the logarithmic derivative of $\Gamma$. For $\lambda$ large and $|\arg \lambda| < \pi$ Stirling's series gives 
$\psi(\lambda/2+1/4) - \log(\lambda/2) = 1/(2\lambda)+ \bigo{ \lambda^{-2}}$. 
So as $\lambda \to \infty$ the leading order terms in $R$ cancel and 
$R(\lambda) =  7/(2\lambda) + \bigo{\lambda^{-1}}$. Inserting this fact into $G(n^{-1}, \lambda)=0$ shows that \eqref{lambda-func-asymp} is the correct asymptotic model. 

The defining relation $\W e^{\W} = 2/(\eps \pi)$ for $\W = \W(2/(\eps \pi))$ implies, by taking logarithms, that $\W^{-1} \log(2/(\eps \pi \W) ) = 1$. 
After some simplification we have 
\[
	\widehat{G}(\eps,\nu) = -2\nu - \frac{2\nu}{\W(2/(\eps \pi))} (1+ \eps \nu) \frac{\log(1+ \eps\nu)}{\eps \nu} 
	-R(\lambda(\eps, \nu)).
\]
Using the fact that $\W(2n/\pi) = \bigo{\log n}$ and computing the derivative of $R$ one may verify that $\widehat{G}(0,0) = 0$ and $\widehat{G}_\nu(0,0) = -2$. Thus, we can apply the implicit function theorem to conclude that a bounded (locally in $\eps$) solution
$\nu = \nu(\eps)$ exists in a neighborhood of $\eps = 0$. 
\end{proof}

Lemma~\ref{lem:crit pts} has the following useful and immediate corollary:
\begin{cor}
For $\lambda = \lambda(n)$ as given in Lemma~\ref{lem:crit pts}
the asymptotic expansion of the phase becomes
\begin{gather}\label{phase expand}
	\phase(z) = 2( \log z + 1 - z ) - \frac{\lambda}{2n} (1 - z + z \log z)
		- \frac{1}{n} \log \zeta \lp \lambda z + \frac{1}{2} \rp
		+ \frac{1}{n} \, \tilde{r}(z; \lambda), \\
\shortintertext{where}
	\nonumber
	\frac{1}{n} \, \tilde{r}(z; \lambda) = 
	\frac{1}{n} \, r(z; \lambda) 
	+ \lp \frac{\lambda}{2n}  \log \frac{\lambda}{2\pi} - 2 \rp (1-z)	
\end{gather}
satisfies the same boundedness conditions \eqref{phase remainder} as the original $r(z;\lambda)$. 
\end{cor}

We complete this section by showing that $\partial_z \phase$ has no other bounded zeros outside the critical strip. 

\begin{lemma}\label{lem:phase'}
	Let $\lambda = \lambda(n)$ be as given by Lemma~\ref{lem:crit pts} and fix $R > \eps > 0$.
	Then for any $z$ such that 
\begin{equation}\label{phase'-1}
	z \in \zerofree 
	\text{\quad and \quad} \eps \leq |z| \leq R
\end{equation}
we have	
\begin{equation}\label{phase'-0}
	\partial_z \phase(z) = 2(z^{-1}-1) + \bigo[R]{ \lp \frac{\log\!\log n}{\log n} \rp^{1/3} }.
\end{equation}
	Additionally, given a fixed $\rho \in (0,1)$, if $|z-1| > \rho$, 
	then there exist $n_0 = n_0(R,\rho) > 0$ such that for all $n > n_0$ 
	we have
\begin{equation}\label{phase'-2}
	\left| \partial_z \phase(z) \right| \geq \rho. 
\end{equation}	
\end{lemma}

\begin{proof}
	Differentiating \eqref{phase expand} one has
	\begin{equation}\label{phase' proof}
		\partial_z \phase(z) =  2(z^{-1} - 1 )  
		  -\frac{\lambda}{n} 
		\frac{ \zeta' ( \lambda z + 1/2 ) }{\zeta(\lambda z + 1/2 )} 
		- \frac{\lambda}{2n} \log z + \frac{1}{n} \partial_z \tilde r(z;\lambda) 
	\end{equation}
Then for any $z$ as described in \eqref{phase'-1} we use \eqref{zeta bounds} to bound the $\zeta'/\zeta$ term in the expression above and note that Lemma~\ref{lem:crit pts} implies that $\lambda/n = \bigo{(\log n)^{-1}}$ to arrive at \eqref{phase'-0}. The last statement follows from the fact that for $|z-1|>\rho$, $2|z^{-1}-1| \geq 2\rho/(1+\rho) = \rho + \frac{\rho(1-\rho)}{1+\rho}$.
Then using \eqref{phase'-0} it is clear that we may choose $n_0(R,\rho)$ such that \eqref{phase'-2} is satisfied whenever $n> n_0$.
\end{proof}

\section{Uniform approximation of \texorpdfstring{$T_{2n}(z)$}{the Taylor polynomials} in the plane}
In this section we construct in a piecewise fashion a uniform approximation of the function $h(z)$ (defined by \eqref{H}).
Inserting this approximation into the representation of the Taylor polynomials $T_{2n}(\lambda z)$ in \eqref{represent} immediately yields a uniform asymptotic representation of the rescaled Taylor polynomials in the plane; this is the result of our Theorem~\ref{thm:T asymp} below. 

Lemma~\ref{lem:crit pts} implies that the contour integral \eqref{H} defining $h$ has two regular stationary points at $z = \pm 1$ and is otherwise non-stationary.
Specifically,
\begin{equation}\label{w relation}
	w^2 = \phase(z) = \frac{\phase''(1)}{2}  (z-1)^2 \lb 1 + \bigo{z-1}  \rb 
\end{equation}
defines a map $w = w(z)$ which, when restricted to any sufficiently small neighborhood $\ball_{1,\delta}$ of $z=1$ (or $\ball_{-1,\delta}$ of $z=-1$), is an invertible conformal map onto a bounded neighborhood of $w = 0$. We choose the branch such that $w$ maps $\dscal_1$ locally to a nearly horizontal contour in the $w$-plane oriented left-to-right:
\begin{equation}\label{zeta local}
	w(z) = -i \sqrt{ \dfrac{ -\phase''(1) }{2} } (z-1) \lb 1 + \bigo{ (z-1) } \rb
	\quad z \in \ball_{1,\delta},
\end{equation}
and enforce symmetry by demanding that $w(z) = w(-z)$ for $z \in \ball_{-1,\delta}$.
The estimate on $\phase''(1)$ in Lemma~\ref{lem:crit pts} implies that 
$\sqrt{-\phase''(1)/2} = 1 + \bigo{1/\log n}$ so that $w=w(z)$ is asymptotically isometric for $z$ near 1 and $n \gg 1$. We fix the neighborhoods $\ball_{\pm1, \delta}$ by requiring that $B_{\pm 1, \delta}$ are, for any sufficiently small $\delta>0$, the two pre-images of the disk of radius $\delta$ in the $w$-plane:
\begin{equation}\label{ball def}
	w \lp \ball_{\pm 1, \delta} \rp = \{ w \in \C \,:\, |w| < \delta \}
\end{equation}
and we let $\ball_\delta = \ball_{1,\delta} \cup \ball_{-1,\delta}$.

For $z$ bounded away from $\pm 1$ a standard stationary phase calculation gives   
\begin{equation}\label{h0}
	h(z) =  h_0(z)\lb 1 + \bigo{n^{-1}}\rb , \qquad h_0(z) = \frac{1}{ \sqrt{2\pi | \phase''(1) |} } \frac{2}{1-z^2}.
\end{equation}
As $z \to \pm 1$ this approximation breaks down as the pole of the integrand in \eqref{H} at $s=z$ approaches the stationary points. At these points a more careful analysis is required which we give below; we prove the following theorem.
\begin{thm}\label{thm:T asymp}
Let $\lambda = \lambda(n)$ be as described in Lemma~\ref{lem:crit pts}, $\chi(z)$ the characteristic function of the set $| \Re z | < 1$, and $h_0(z)$, defined by \eqref{h0},  the leading order stationary phase approximation of $h(z)$. 
Then as $n \to \infty$ the Taylor polynomials described by \eqref{represent} admit the asymptotic expansion
\begin{multline}\label{T asymp}
	T_{2n-2}( \lambda z) = T_{2n-1}( \lambda z) =
	\begin{dcases}
		f( \lambda z) \lb 
		\chi(z) 
		- \frac{ e^{n \phase(z)} }{ \sqrt{n} } h_0(z) \lp  
		1 + \Ecal(z) \rp
		\rb
		& z \in \C \backslash \ball_\delta 
		\\ \bigskip
		f( \lambda z) 
		\lb 
		\frac{1}{2} \erfc(i \sqrt{n} w(z) ) 
		- \frac{e^{n\phase}(z) }{ \sqrt{n} } \Ecal(z)
		\rb
		& z \in  \ball_\delta.
	\end{dcases}
\end{multline}
where the residual error function $\Ecal(z)$ is bounded, analytic in $\C \backslash \lp (\dscal_1 \backslash \ball_\delta) \cup \partial \ball_\delta \rp$, and satisfies
\begin{equation}
	\Ecal(z) = \begin{dcases}
		\bigo{n^{-1}} 
			& z \in \ball_\delta^c \\
		h_0(z) + \frac{1}{2i \sqrt{\pi} w(z)} + \bigo{n^{-1}}
			& z \in \ball_\delta.
	\end{dcases}
\end{equation}
\end{thm}

\begin{cor}\label{cor:1}
	Let $\lambda = \lambda(n)$ be as described in Lemma~\ref{lem:crit pts} and $e^{n\phase(z)}$ be as defined in \eqref{phase}. Define
		\[
		\begin{aligned}
		\Sz &= 
		\left\{ 
		z \in \C \,:\, \left| \Re z \right| < 1 \text{ and } 
		\left| e^{\phase(z)} \right|  < 1 \right\}, \\ 
		\nSz_{-} &= 
		\left\{
		z \in \C \,:\, \left| \Re z \right| > 1 \text{ and }
		\left| e^{\phase(z)} \right| < 1  \right\}, \\
		\nSz_{+} &= 
		\left\{
		z \in \C \,:\, 
		\left| e^{\phase(z)} \right| > 1 \right\}.
		\end{aligned}
	\]
	Then the relative error satisfies 
	\begin{align*}
		\lim_{n \to \infty} 
		\left| \frac{ T_{2n}( \lambda z) }{ \xi(1/2 + \lambda z) } - 1 \right |
		 &= 0 
		 \qquad 
		 \ z \in \Sz, \\
		\lim_{n \to \infty}
		\left| \frac{ T_{2n}( \lambda z) }{ \xi(1/2 + \lambda z) } \right |
		&= \begin{cases}
			0 & z \in \nSz_{-} \\
			\infty & z \in \nSz_{+}
		\end{cases}
	\end{align*}		
\end{cor}

Let us begin to develop the tools to prove Theorem~\ref{thm:T asymp}.
For $z \in B_\delta$, we define the function 
$k: B_\delta \backslash \partial \mathcal{S}_1 \to \C$ by 
\begin{gather}\label{k def 1}
	k(z) = \sqrt{n} \hat k ( \sqrt{n} w(z)),  \quad \hat k(s) =\frac{1}{2\pi i} 
	\int_\g e^{-t^2} \frac{ \dd t}{t - s} 
\end{gather}
where $\g$ is the left-to-right oriented contour passing through the origin formed by extending the scaled image $\sqrt{n} w( \dscal_1 \cap \ball_{1,\delta} )$ horizontally to infinity in both directions. We will show that this function well approximates $h(z)$ in $\ball_\delta$. 
For our purposes, the essential fact is that $k(z)$ is analytic in $\ball_\delta \backslash \dscal_1$ and satisfies the same jump relation on $\dscal_1$ as the function $h(z)$ which we are attempting to approximate:
\begin{equation}\label{HK local jumps}
	h_+(z) - h_-(z) =  \sqrt{n} e^{ -n \phase(z)} 
	= \sqrt{n} e^{-n w(z)^2} = k_+(z) - k_-(z)
		\qquad 
	z \in \dscal_1 \cap \ball_\delta. 
\end{equation}

The integral defining $\hat k$ can be explicitly evaluated:  
integrating by parts one easily shows that $\hat{k}$ satisfies  
$
	\hat{k}' + 2s \hat{k} = i/ \sqrt{\pi};
$
using \eqref{k def 1} and the residue calculus one sees that $\hat k_\pm(0) = \pm 1/2$. Solving the differential equation for $\hat{k}$ yields, upon composition with 
$\sqrt{n} w(z)$:
\begin{equation}\label{k}
	k(z) = \sqrt{n} e^{-n \phase(z)} \lb \chi(z) - \frac{1}{2} \erfc( i \sqrt{n} w(z) ) \rb. 
\end{equation}

Using the known asymptotic behavior \cite[eq.~7.12.1]{DLMF}  of the complementary error function
\begin{equation}\label{erfc asymp}
	e^{s^2} \erfc(s) \sim 
	\begin{dcases}
		\frac{1}{\sqrt{\pi} s} 
		\sum_{m=0}^\infty (-1)^m \frac{  \Gamma(1/2 + m) }{\Gamma(1/2)} s^{-2m} 
		& \left| \arg(s) \right | < 3\pi/4 \\
		2e^{s^2} +  \frac{1}{\sqrt{\pi} s}  
		\sum_{m=0}^\infty (-1)^m \frac{  \Gamma(1/2 + m) }{\Gamma(1/2)} s^{-2m}
		& \left| \arg(-s) \right | < 3\pi/4 ,
	\end{dcases}
\end{equation}
it follows that uniformly in the $s$-plane
\begin{equation}\label{K asymptotics}
	\hat{k}(s)  = -\frac{ 1 }{2 i\sqrt{\pi} s} \lb 1 + \bigo{s^{-2}} \rb 
	\qquad
	s \to \infty.
\end{equation}

Putting together the steepest descent approximation \eqref{h0}, valid in $\ball_\delta^c$, and our local model $k$ we define the residual error function
\begin{equation}\label{R}
	\Ecal(z) = \begin{cases}
		h(z) - h_0(z) & z \in \ball_\delta^c \\
		h(z) - k(z) & z \in \ball_\delta.
	\end{cases}
\end{equation}	
Orienting the contour $\partial \ball_\delta$ counterclockwise we have the following lemma.
\begin{lemma}\label{lem:r}
The residual $\Ecal(z)$, defined by $\eqref{R}$, is analytic in $\C \backslash \Gamma_\Ecal$, 
$\Gamma_\Ecal = (\dscal_1 \backslash \ball_\delta) \cup \partial \ball_\delta$, and given by
\begin{gather}
	\label{r cauchy}
	\Ecal(z) = \frac{1}{2 i \pi} \int_{\Gamma_\Ecal} \frac{ v_\Ecal(w) }{w-z} \, \dd w, \qquad
	v_\Ecal(z) = 
	\begin{cases}
	 	h_0(z) - k(z) & z \in \partial B_\delta \\
	 	\sqrt{n} e^{-n \phase(z)} & z \in \contour \backslash \ball_\delta.
	 \end{cases}
\end{gather}
Moreover, there exist $n_0, \delta_0 > 0$ such that for any $n \geq n_0$ and $\delta \leq \delta_0$ 
\begin{equation}\label{r bounds}
	\Ecal(z) = 
	\begin{dcases}
		\lb h_0(z) + \frac{1}{2i\sqrt{\pi} w(z)} \rb 
		+ \bigo{ n^{-1} } & z \in \ball_\delta \\
		\bigo{ n^{-1} } & z \in \ball_\delta^c
	\end{dcases}
\end{equation}
uniformly for $z$ in each set. 
\end{lemma}

\begin{proof}
From \eqref{H} and \eqref{h0} we see that $\Ecal(z)$ is analytic in $\ball_\delta^c$ except along $\dscal_1$ where it inherits the jump discontinuity of $h$ and that it vanishes as $z\to \infty$. 
Inside $\ball_\delta$, \eqref{HK local jumps} implies that $\Ecal$ is continuous across $\partial \mathcal{S}_1$, and hence is analytic. 
The jump $v_\Ecal$ and Cauchy integral representation of $\Ecal(z)$ in \eqref{r cauchy} follow immediately. 

The bounds in \eqref{r bounds} follow from two observations:  
first, the jump  $v_\Ecal$ is exponentially small on $\dscal_1 \backslash \ball_\delta$, 
specifically $v_\Ecal(z) = \bigo{ e^{-cn} }$ for $z \in \dscal_1 \backslash \ball_\delta$ where $c = \min\limits_{y \in [\delta,\infty)} \Re \,\phase(1+ i y) > 0$; 
and secondly, on the disk boundary $\partial \ball_\delta$ we have 
\[
	v_\Ecal(z)  
	= \lb  h_0(z) + \frac{1}{2i \sqrt{\pi} w(z)}\rb 
	-  \lb k(z) + \frac{1}{2i \sqrt{\pi} w(z)} \rb.
\] 
Using \eqref{h0} and \eqref{zeta local} it's easy to see that the first bracketed term has vanishing residues at $z = \pm 1$; it therefore extends to a bounded analytic function for $z \in \ball_\delta$.
The second bracketed term is not analytic in $\ball_\delta$, but using \eqref{k def 1} and \eqref{K asymptotics} it admits a Laurent expansion on $\partial \ball_\delta$ which is uniformly $\bigo{n^{-1}}$. 
Thus, the Cauchy transform of the first bracketed term can be explicitly evaluated 
for any $z \in \C \backslash \partial \ball_\delta$ by the Cauchy integral formula; using the boundedness of the Cauchy projection operators the Cauchy transform of the second bracketed term above is everywhere $\bigo{n^{-1}}$. The expansion \eqref{r bounds} follows immediately.
\end{proof}

\begin{proof}[Proof of Theorem~\ref{thm:T asymp}]
Equation \eqref{R} and Lemma~\ref{lem:r} yield an asymptotic expansion of $h(z)$ in $\ball_\delta$ and $\ball_\delta^c$. Plugging this result into \eqref{represent} gives \eqref{T asymp} which completes the proof.  
\end{proof}

\section{Counting the zeros of the taylor polynomials}
\label{sec:taylor zeros}
\begin{figure}[htb]
	\centering
	\includegraphics[width=.9\textwidth]{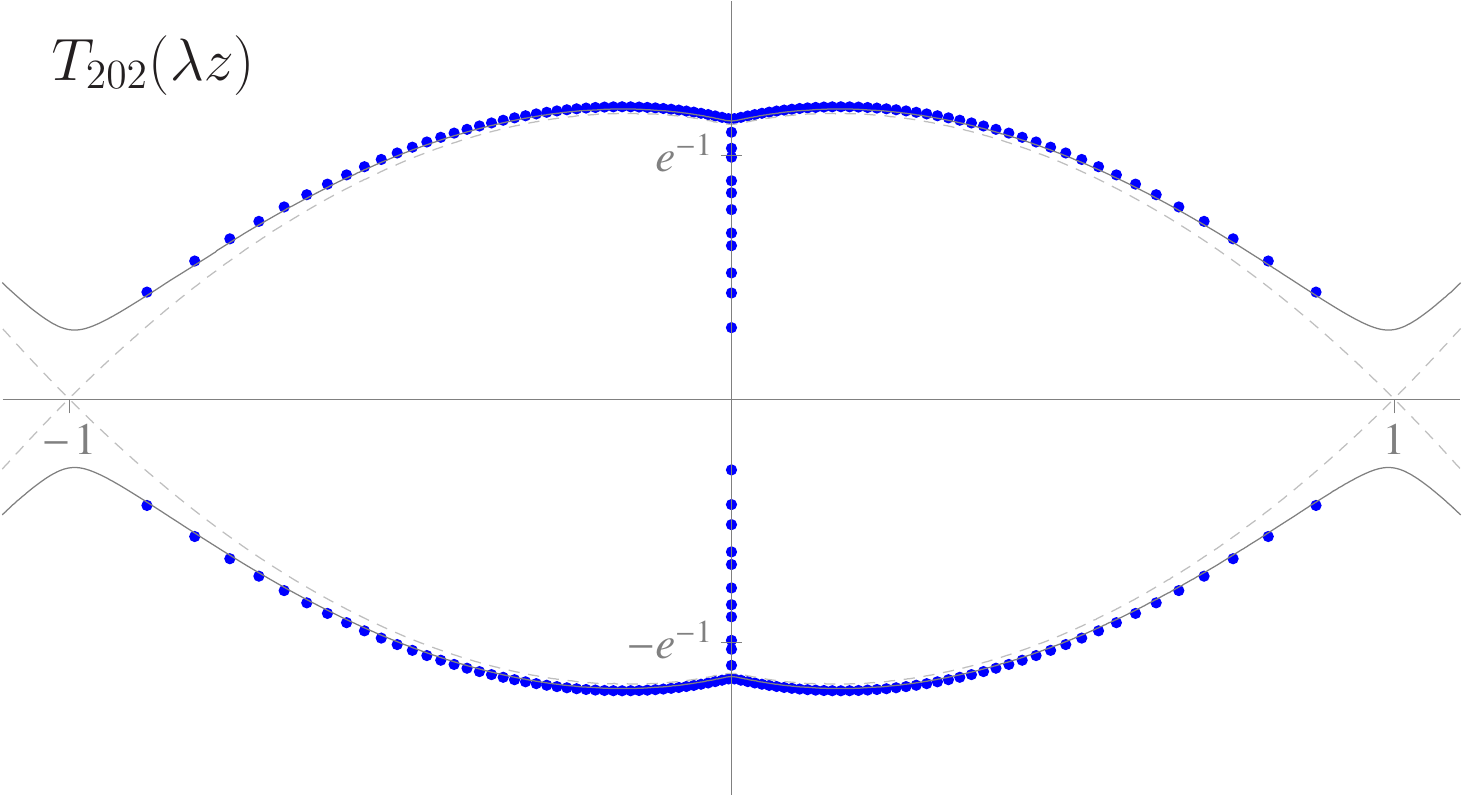}
	\caption{
	Zeros of $T_{202}(\lambda z)$, the $202^{\mathrm{nd}}$ degree Taylor polynomial 
	of $\xi(z+1/2)$ in the rescaled plane, computed using \cite{Mathemagica}. 
	The scaling parameter $\lambda$ is given by Lemma~\ref{lem:crit pts}. 
	As $n \to \infty$, spurious zeros approach the level curve $D_n^{(0)}$ (dashed line);
	for finite $n$, the improved curve $D_n^{(1)}$ (solid line) more accurately 
	approximates zeros.  
	A particular fraction lie inside the curve $D_n^{(0)}$. These Hurwitz zeros converge to 
	shifted and scaled nontrival roots of the $\zeta$ function. 
	The difference between the 11 numerically computed zeros of $T_{202}(\lambda z)$ on the positive critical line and the (rescaled) first 11 nontrivial zeros of $\zeta$ are given in Table~\ref{table:zeros}. 
	\label{fig:zeta}
	}
\end{figure}
Here and in what follows, $\lambda = \lambda(n)$ is as described by Lemma~\ref{lem:crit pts}, so that in particular $\lambda$ satisfies $4n - \lambda \log \lp \frac{\lambda}{2\pi} \rp = \bigo{1}$.
It follows from \eqref{represent} that any zero $z_{k,n}$ of $T_{2n-2}(\lambda z)$ satisfies
\begin{equation}\label{zero condition 1}
	\mathcal{G}(z_{k,n}) = \frac{\log n}{2n} + \frac{2k\pi i}{n} , 
	\qquad
	\mathcal{G}(z):= \phase(z) + \frac{1}{n} \log h(z) , 
\end{equation}
This can happen in one of two ways, either: 
a) $z_{k,n}$ is a Hurwitz zero converging to a zero of $f(\lambda z)$---and thus lies inside the rescaled critical strip; 
or b) $z_{k,n}$ is a spurious zero, and does not approach a root of $f(\lambda z)$;  in both cases $z_{k,n}$ must approach the level curve
\begin{equation}
	D_n^{(0)} = \{ z \in C \,:\, \Re \phase(z) = 0, \text{ and } |z|<1 \},
\end{equation}
which is nearly\footnote{we do not call $D_n^{(0)}$ the actual Szeg\H{o} curve because $\phase$ still has weak $n$ dependence. Strictly speaking, the Szeg\H{o} curve should be defined as $D_\infty = \lim_{n \to \infty} D_n^{(0)}$.}
the Szeg\H{o} curve for $f$.  Although not necessary for this paper, this set can be shown to consist of a collection of disjoint components collapsing upon Hurwitz zeros (and the corresponding zeros of the function $f(\lambda z)$) together with an additional large component attracting those zeros that are spurious.

Let
\begin{equation}
	\mathcal{Z}_n = \left\{ z_{k,n} \,:\, T_{2n-2}(\lambda z_{k,n}) = 0 
	\text{ and } \Re z_{k,n} > \frac{1}{2\lambda} 
	\right\}
\end{equation}
denote the set of (spurious) zeros of $T_{2n-2}(\lambda z)$ outside the rescaled critical strip. The results in this section culminate in the following theorem:

\begin{thm}\label{thm:zeros-outside}
Let $T_{2n-2}(\lambda z)$ be the rescaled Taylor polynomial of degree $2n-2$ defined by 
\eqref{represent} and Lemma~\ref{lem:crit pts}. Then as $n \to \infty$
\[
	| \mathcal{Z}_n | =
	n - \frac{\lambda \y}{2\pi} \log \lp \frac{\lambda \y}{2\pi} \rp
	+ \frac{\lambda \y}{2\pi} - \frac{1}{4\pi \y} \log \lp \frac{\lambda \y}{2\pi} \rp
	+ \bigo{\log\!\log \lambda \y}.
\]
Here $\y$, defined in Lemma~\ref{lem:outer zeros} below, is the imaginary part of a point on $D_n^{(1)}$---a further improvement to the curve $D_n^{(0)}$---at the edge of the critical strip. 
\end{thm} 

As $T_{2n-2}$ has exactly ${2n-2}$ zeros this has the immediate and obvious corollary:
\begin{cor}\label{cor:taylor zeros inside}
As $n \to \infty$, the Taylor polynomial $T_{2n-2}(\lambda z)$ has
\[
	\frac{\lambda \y}{2\pi} \log \lp \frac{\lambda \y}{2\pi} \rp
	- \frac{\lambda \y}{2\pi} + \frac{1}{4\pi \y} \log \lp \frac{\lambda \y}{2\pi} \rp
	+ \bigo{\log\!\log \lambda \y}
\]
zeros in the rescaled critical strip.
\end{cor}

\begin{rem}  
Well known estimates on the behavior of $\zeta$ within the critical strip show that the level set $\Re \mathcal{G} = \frac{\log n}{2 n}$, on which all zeros of $T_{2n-2}(\lambda z)$ must live, remains within a rectangle whose height is bounded by ${\y} + \bigo{ \frac{\log{n}}{n}} $. Corollary~\ref{cor:taylor zeros inside} is therefore consistent with the Riemann-von Mangoldt formula \eqref{Riemann-Mangoldt} using $T= \lambda {\y}$.  The precision of the error bound for the zeros of the Taylor polynomials suggests that there are a growing number of spurious zeros within the rescaled critical strip.
\end{rem}

Theorem~\ref{thm:zeros-outside} is proved below using the asymptotic representation in Theorem~\ref{thm:T asymp}. 
We first count those zeros bounded away from the stationary points $z=\pm 1$ by constructing a set of approximate zeros $\alpha_{k,n}$ and then demonstrating that each of these is in one to one correspondence with an actual zero $z_{k,n}$ of the Taylor polynomial in the zero free region $\zerofree$. 
We then count the zeros near each of the stationary points using a Rouche theorem type argument. 
Finally, note that the four-fold symmetry $T_{2n}(z) = T_{2n}(-z) = T_{2n}(z^*)^*=T_{2n}(-z^*)^*$ implies that it is sufficient to study only those zeros in the closed positive quadrant: $\Re z, \Im z \geq 0$.

\subsection{Number of zeros outside the critical strip, away from the stationary points}
\label{sec:bulk zeros}

Let 
\begin{equation}\label{U-def}
	\U =  \left\{ 
	z = \overline{\C^+} \,:\, \frac{1}{2\lambda} \leq \Re z \leq 1  
	\text{ and } z \nin \ball_{1, \delta} 
	\right\}.
\end{equation}
denote the vertical strip in $\C^+$ between the critical strip and the stationary point at $z=1$ with a small neighborhood of $z=1$ deleted. Both $f(\lambda z)$ and $h(z)$ are analytic and zero free in $\U$, so $\phase(z)$ and $\log h(z)$ are each well defined (we choose the branches real valued for $z \in \U \cap \R$).

As a first step toward Theorem~\ref{thm:zeros-outside} we want to estimate the number of zeros of $T_{2n-2}(\lambda z)$ in $\U$. We could approximate the zeros $z_{k,n}$ by points along $D_\infty$, but for our purposes it will be more convenient to work with 
\begin{equation}
	\begin{gathered}
	D_n^{(1)} = 
	\{ z \in C \,:\, \Re \mathcal{G}_1(z) = \tfrac{\log n}{2n}, \text{ and } |z|<1 \},\\
	\mathcal{G}_1(z):= \phase(z) + \frac{1}{n} \log h_0(z)
	\end{gathered}
\end{equation} 
which is the (first) correction to the level curve $D_n^{(0)}$ that better attracts the spurious zeros, analogous to the improved Szeg\H{o} curve \eqref{exp Szego1}, which comes from keeping the first term in the asymptotic series for $h$.   
We define the \textit{approximate zeros}, $\alpha_{k,n}$, as roots of the equation.
\begin{gather}\label{zero condition 2}
	\mathcal{G}_1(\alpha_{k,n}) = \frac{\log n}{2n} + \frac{2k\pi i}{n} , 
\shortintertext{and denote by $\Acal_n$ the set of approximate zeros of $T_{2n-2}(\lambda z)$ which lie in $\U$:}
	\Acal_n = \left\{ \alpha_{k,n} \in \U  \,:\, 
	\mathcal{G}_1(\alpha_{k,n}) = \frac{\log n}{2n} + \frac{2k \pi i}{n} 
	\right\}.
	\label{approx zeros}
\end{gather}

We begin by describing the shape of the improved level curve $\mathcal{D}_n^{(1)}$ along which our approximate zeros accumulate in the region
\[
	\zerofreeone = 
	\left\{ z \in \zerofree \backslash \ball_{1,\delta}
	\,:\, 
	 \Re z \in [0,1] 
	\right\}
\]
		
\begin{lemma}\label{lem:z crit} 
Let $z = x+ i y$. Fix $A>0$ defining the zero free region $\zerofree$.
Then there exist $n_0>0$ such that for any $n>n_0$ the level curve $\Re \mathcal{G}_1 = \log n/(2n)$ implicitly defines a single smooth non-intersecting curve $y=Y(x)$ for $z \in \zerofreeone$ as defined above. Near the edge of the critical strip, that is for,
\begin{gather*}
	-\frac{A}{\lambda (\log \lambda)^{2/3} (\log\log \lambda)^{1/3}}  
	< x - \frac{1}{2\lambda} < \frac{A}{\lambda},
	\shortintertext{the curve $y=Y(x)$ satisfies}
	Y(x) = \frac{8n}{\pi \lambda} \mathop{W} \lp \frac{ \pi \lambda}{8n} 
	e^{-1 + x + (\lambda+ \log n)/(4n) } \rp + \bigo{ \frac{ \log\!\log n}{n} },
\end{gather*}
where $\mathop{W}$ is the Lambert-W function. 
\end{lemma}

\begin{proof}
Both $\phase$ and $h_0$ are analytic in $\zerofreeone$. Lemma~\ref{lem:phase'} bounds $|\partial_z\phase|$ below uniformly in $n$, and $h_0$ has a bounded derivative in $\zerofreeone$. 
It follows that for all sufficiently large $n$, $\partial_z \mathcal{G}_1(z) \neq 0$ for all $z \in \zerofreeone$ and thus the level set $\Re \mathcal{G}_1 = \log n/(2n)$ must consist of a collection of smooth nonintersecting arcs in $\zerofreeone$ with no finite endpoint in $\zerofreeone$. 

As $\lim_{z \to 0} \Re \phase(z) = -\infty$, $\phase(1) = 0$, and $\partial_z \phase$ has no zeros on $(0,1)$, $\Re \phase < 0$ for $z \in (0,1)$. Thus, for any $x_0 \in (0,1)$, for all sufficiently large $n$, $\Re \mathcal{G}_1 < 0$ for all $x \in (0, x_0)$. So no branch of the level curve may leave $\zerofreeone$ through the real axis.
Away from $z = 0$ we use \eqref{phase expand} to write
\begin{multline}\label{re phase}
	\Re \mathcal{G}_1(z) = 2(\log |z| + 1 - x) 
	- \frac{\lambda}{2n}(1-x +x \log|z| - y \arg(z)) \\
	-\frac{1}{n} \log \left| \zeta(\lambda z + 1/2) \right| 
	+ \frac{1}{n}\Re \tilde{r}(z; \lambda) + \frac{1}{n} \log | h_0(z) | .
\end{multline}
From this expansion we observe that: (1) the level curves are bounded above since $\Re \mathcal{G}_1(z)$ grows without bound as $y \to \infty$ with $x$ bounded; (2) for any $y_0>0$, if $z = 1 + i y$, with $y > y_0$, then for all $n$ large enough $\Re \mathcal{G}_1(z) > c(y_0) > 0$. 
So all branches of the level curve $\Re \mathcal{G}_1 = \tfrac{\log n}{2n}$ in $\zerofreeone$ must enter $\zerofreeone$ through it's left edge and leave by entering $\ball_{1,\delta}$.

Since all branches of the level set are bounded away from the origin and infinity, for any $ z = x +i y$ along the level set $\Re \phase(z) = \log n/(2n)$ with $\Re z = x < A/\lambda$ we have: 
\begin{equation}\label{small x expansions}
	\log |z| = \log y + \bigo{\lambda^{-2}},
	\qquad \qquad 
	\arg (z) = \frac{\pi}{2} - \frac{x}{y} + \bigo{\lambda^{-3}}.
\end{equation}
Inserting these into \eqref{re phase} one has
\begin{gather*}
	\Re \phase(z) = g(x,y) + \bigo{\frac{ \log\!\log n}{n} },	
	\qquad
	g(x,y) = 2(\log y +1-x) - \frac{\lambda}{2n} (1 - \frac{\pi}{2} y),
\end{gather*}
where we've used \eqref{zeta bounds} to bound $\log \zeta(\lambda z+1/2)$.
For each $0 < x < A/\lambda$ there is a single solution $y$ of $g(x,y) = (\log n)/(2n)$. It follows that there is only a single branch of the level curve $\Re \mathcal{G}_1 = \tfrac{\log n }{2n}$ in $\zerofreeone$. One may then solve $g(x,y) = (\log n)/(2n)$ for $y$ using the Lambert-$W$ function, which gives the leading term of $y=Y(x)$ for $x \in \zerofreeone$ with $\Re x < A/\lambda$. The error bound is immediate. 
\end{proof}

\begin{lemma}\label{lem:outer zeros}
As $n \to \infty$, the number of approximate zeros in $\U$ satisfies 
	\[
		\left| \Acal_n \right| = \frac{n}{2} 
		- \frac{\lambda \y}{4\pi} \log \lp \frac{ \lambda \y}{2\pi} \rp  
		+ \frac{\lambda \y}{4\pi } 
	 	- \frac{1}{8\pi \y} \log \lp \frac{ \lambda \y}{2\pi} \rp  
		- \frac{n \delta^2}{2\pi} +\bigo{\log\!\log \lambda \y}.
	\]
Here, $\y = Y \lp \tfrac{1}{2\lambda} \rp$, with $Y(x)$ as described by Lemma~\ref{lem:z crit}, is the imaginary part of $z$ where the level curve $\Re \mathcal{G}_1 = \tfrac{\log n}{2n}$ meets the edge of the critical strip, and $\delta$ is the radius of the image-disk $w( \ball_{1,\delta})$ in the $w$-plane. 
\end{lemma}

\begin{proof}
Moving along the level curve $\Re \mathcal{G}_1 = \tfrac{\log n }{2n}$ from $\Re z = 1$ towards the critical strip, $\Im \mathcal{G}_1$ is strictly increasing, so denoting by $z_0$ the point at which $\Re \mathcal{G}_1 = \tfrac{\log n }{2n}$ intersects the boundary of $\ball_{1,\delta}$, and noting that $\log h_0$ is bounded outside $\ball_{1,\delta}$,
the number of approximate zeros in $\U$ is given by
\[
	\frac{n}{2\pi} \lb 
	\Im \mathcal{G}_1 \lp \frac{1}{2\lambda} + i \y \rp - 
	\Im \mathcal{G}_1 \lp z_0 \rp
	\rb
	=
	\frac{n}{2\pi} \lb 
	\Im \phase \lp \frac{1}{2\lambda} + i \y \rp - 
	\Im \phase \lp z_0 \rp
	\rb + \bigo{1}.
\]

Recall that the set $\ball_{1,\delta}$ is chosen such that the image $w(\ball_{1,\delta})$ under the map $w = w(z)$ defined by \eqref{w relation}-\eqref{zeta local} is a disk of radius $\delta$. The condition that $\Re \phase(z_0) = \Re w^2  = \delta^2 \cos(2\arg w) = (\log n)/2n$ gives 
\[
	\Im \phase(z_0) =  \delta^2 \sin(2\arg w) 
	= \delta^2 - \bigo{(\log n)^2/(n\delta^2)^2}. 
\]
Similarly, using Lemma~\ref{lem:arg bound}, \eqref{phase expand} and \eqref{small x expansions}
we have
\begin{align*}
	\Im \phase \lp \frac{1}{2\lambda} + i \y \rp 
	&= 2 \lb \frac{\pi}{2} - \y - \frac{1}{2\lambda \y}  \rb 
	+ \frac{\lambda}{2n} \lb \y - \y \log \y  \rb  
	+ \bigo{ \frac{\log\log \lambda \y}{n} } \\
	&= \pi - \frac{\lambda \y}{2n} \log \lp \frac{ \lambda \y}{2\pi} \rp 
	+ \frac{ \lambda \y}{2n} - \frac{1}{4n \y} \log \lp \frac{ \lambda \y}{2\pi} \rp
	+ \bigo{ \frac{\log\log \lambda \y}{n} }
\end{align*}
where we have used the estimate $\lambda = \bigo{n/\log n}$ implied by Lemma~\ref{lem:crit pts} to drop lower order terms and in the last equality we've used the first asymptotic statement in Lemma~\ref{lem:crit pts} to simplify. The result follows immediately. 
\end{proof}

\begin{lemma}\label{lem:true zeros outside} 
Fix $A>0$ to define a zero free region $\zerofree$ as in \eqref{zero-free}. 
Let $T_{2n-2}(\lambda z)$ be the rescaled Taylor polynomial of degree $2n-2$ defined by \eqref{represent} and Lemma~\ref{lem:crit pts}; 
let $z_{k,n}$ and  $\alpha_{k,n}$ denote actual and approximate zeros of $T_{2n-2}(\lambda z)$ defined by \eqref{zero condition 1} and \eqref{zero condition 2} respectively.  
Then for all sufficiently large $n$, each approximate zero $\alpha_{k,n} \in \zerofree$ corresponds to a distinct zero $z_{k,n}$ of $T_{2n-2}(\lambda z)$. Moreover, 
\[
	|z_{k,n} - \alpha_{k,n}| = \bigo{ n^{-2} }.
\]
\end{lemma}

\begin{proof}
Fix $n$. 
Clearly for $k \neq \ell$, any solutions $z_{k,n}$ and $z_{\ell,n}$ of \eqref{zero condition 1} in the zero free region of $f$ are distinct as each corresponds to a distinct value of the single valued function $\Im(\phase(z)-n^{-1}\log h(z))$. 
Using \eqref{zero condition 2}, the root condition \eqref{zero condition 1} can be rewritten in the form $\mathfrak{G}_k( z_{k,n}, n^{-1} ) = 0$ where 
\begin{align*}
	\mathfrak{G}_k(z,\eps) 
	&= \phase(z) - \phase(\alpha_{k,n}) + \eps \lb \log h(z) - \log h_0(\alpha_{k,n}) \rb.
\end{align*}
Now, $\mathfrak{G}_k(\alpha_{k,n}, 0) = 0$ and Lemma~\ref{lem:phase'} guarantees that for all sufficiently large n (independent of k), $|\partial_z\mathfrak{G}_k(\alpha_{k,n},0)| =  |\partial_z \phase(\alpha_{k,n})| \geq \rho > 0$ for all sufficiently large $n$. 
Invoking the implicit function theorem, there exist a unique solution $z_{k,n}$ of \eqref{zero condition 1} for all sufficiently large $n$ in a neighborhood of $\alpha_{k,n}$. Expanding, we have (again uniformly in k)
\[
	z_{k,n} = \alpha_{k,n} - \frac{1}{n} 
	\frac{ \log h(\alpha_{k,n}) - \log h_0(\alpha_{k,n})}{\partial_z \phase(\alpha_{k,n})} 
	+ \bigo{\frac{1}{n^2}}.
\]
Recalling \eqref{h0}, we observe that $\log h(\alpha_{k,n}) - \log h_0(\alpha_{k,n}) = \bigo{n^{-1}}$ which completes the result.
\end{proof}

\subsection{Number of zeros near the stationary points}
\label{sec:saddle zeros}

Near the stationary points $z = \pm 1$ the zeros of $T_{2n}(\lambda z)$ are not spaced uniformly along the level curve $D_n^{(1)}$.  
Theorem~\ref{thm:T asymp} suggest that the zeros of $T_{2n-2}(\lambda z)$ should be well approximated by the zeros of $g(i n^{1/2} w(z))$. 
Recall that $\ball_\delta$ is chosen such that the scaling map $\up(z) = i n^{1/2} w(z)$ 
maps $\ball_\delta$ to a disk of radius $n^{1/2} \delta$ centered at the origin in the $\up$-plane, \ie, $\up(\ball_{1,\delta}) = \mathbb{D}(0, n^{1/2} \delta)$.
The zeros of $\erfc(\upsilon)$ are well known and come in conjugate pairs \cite[\S 7.13(ii)]{DLMF}. 
Enumerating the zeros of $\erfc(\upsilon)$ in $\C^+$ by $\upsilon_k =\mu_k + i \nu_k$, according to increasing absolute value, the large modulus zeros of $\erfc(\upsilon)$ are asymptotically given, for $k \gg 1$, by
\begin{gather}\label{erfc zeros}
	\upsilon_k := \mu_k + i \nu_k \\
	\nonumber
	\begin{aligned}
		\mu_k &= - \varsigma + \frac{1}{4} \tau \varsigma^{-1} 
			- \frac{1}{16}(1-\tau  + \frac{1}{2} \tau^2) \varsigma^{-3} + \dots 
		&&\qquad \varsigma = \sqrt{(k-1/8) \pi}	
			\\
		\nu_k &=  \varsigma + \frac{1}{4} \tau \varsigma^{-1} 
			+ \frac{1}{16}(1-\tau + \frac{1}{2} \tau^2) \varsigma^{-3} + \dots 
		&& \qquad \tau = \log \lp 2 \varsigma \sqrt{2\pi} \rp	
	\end{aligned}
	\shortintertext{from which it follows that }
	\label{erfc zero size}
	|\upsilon_k|^2 = 2\pi(k-1/8) + \bigo{ k^{-1} \log(k)^2}.
\end{gather}	

To count the number of zeros of $T_{2n-2}(\lambda z)$ in $B_{\pm 1,\delta}$ we first introduce the integer valued functions
\begin{equation}\label{Kcount}
	\K^-(n,\delta) := \floor*{ \frac{n \delta^2}{2\pi} - \frac{3}{8} }
	\quad \text{and} \quad 
	\K^+(n,\delta) := \ceil*{ \frac{n \delta^2}{2\pi} - \frac{3}{8} },
\end{equation}
where $\floor{x}$ and $\ceil{x}$ are the floor and ceiling functions respectively. 
\begin{lemma}
\label{lem:local zeros}
There exist $\delta_0 > 0$ such that for any fixed $\delta$,  $0 < \delta < \delta_0$, there exist $n_0(\delta_0, \delta)$ such that for any $n > n_0$ 
the Taylor polynomial $T_{2n-2}(\lambda z)$ has either $2\K^-(n,\delta)$ or $2\K^+(n,\delta)$ zeros in $\ball_{\pm 1,\delta}$.
\end{lemma}

\begin{proof}
Due to even symmetry of $T_{2n-2}(\lambda z)$ we consider only $\ball_{1,\delta}$.
For simplicity, temporarily let $\up(z) = i n^{1/2} w(z)$ and write $g(\up) =  \frac{1}{2} e^{\up^2} \erfc(\up)$. 
As $f(\lambda z)$ is zero free in $\ball_{1,\delta}$, the representation \eqref{T asymp} implies that $T_{2n-2}(\lambda z)$ has the same number of zeros in $\ball_{1,\delta}$ as the function $g(\up(z)) + n^{-1/2} \Ecal(z)$. 

For a fixed choice of $\delta > 0$ define the radii 
\[
R^{-}_{\delta} = \sqrt{ 2\pi \lp \K^-(n,\delta) + \frac{3}{8} \rp },
	\qquad 
R^{+}_{\delta} = \sqrt{ 2\pi \lp \K^+(n,\delta) + \frac{3}{8} \rp }	.
\]
The proof follows from Rouche's theorem.
From \eqref{erfc zero size} it follows that there are exactly 
$2 \K^-(n,\delta)$ zeros of $g(\up)$ in $\mathbb{D}(0, R_\delta^-)$ 
and
$2 \K^+(n,\delta)$ zeros of $g(\up)$ in $\mathbb{D}(0, R_\delta^+)$ 
for any  sufficiently large $n$. 
Lemma~\ref{lem:r} guarantees for all $z \in \up^{-1}(\mathbb{D}(0, R_\delta^+))$ that $| \Ecal(z)| < C_0 $ for some fixed positive constant $C_0>0$ 
independent of $\delta$\footnote{
We are slightly abusing notation here, in \eqref{R} $\Ecal(z)$ is piecewise defined inside and outside $\ball_\delta$. What we mean here is the analytic extension of $\Ecal(z)$ from inside $\ball_{1,\delta}$ to a set containing $\up^{-1}(\mathbb{D}(0,R_\delta^+))$ which can be see to exist simply by deforming the contour $\Gamma_\Ecal$ in \eqref{r cauchy}.  
}. 
As we will show below, there also exist a constant $C_1>0$ such that $|g(\up)| > C_1 / R^\pm_\delta$ on the circles of radii $R_\delta^\pm$. 
By choosing $\delta$ such that $R^{\pm}_\delta / n^{1/2}  < \delta_0 = C_1/C_0$, Rouche's theorem implies that $T_{2n-2}(\lambda z)$ has $2\K^-(n,\delta)$ zeros in $\up^{-1}(\mathbb{D}(0, R_\delta^-))$ and $2\K^+(n,\delta)$ zeros in $\up^{-1}(\mathbb{D}(0, R_\delta^+))$. 
The result then follows from observing that $\up^{-1}(\mathbb{D}(0, R_\delta^-)) \subseteq \ball_{1,\delta} \subseteq \up^{-1}(\mathbb{D}(0, R_\delta^+))$, so we have determined the number of zeros $T_{2n-2}(\lambda z)$ in $\ball_{1,\delta}$ to within two, depending on the location of the extra pair of zeros in $\up^{-1}(\mathbb{D}(0, R_\delta^+))$. 

It remains to show that $|g(\up)| > C_1/R_\delta^\pm$ for $|\up| = R_\delta^\pm$. 
For any $\delta>0$, $R_\delta^\pm \sim \delta n^{1/2}$ so for all sufficiently large $n$ the asymptotic series for $g(\up)$ given in \eqref{erfc asymp} can be applied. 
Away from the rays $\arg (\up) = \pm 3\pi/4$ the desired bound is immediate; when one needs the second expansion in \eqref{erfc asymp} the exponential term is either beyond all orders small or dominant away from these rays---in either case the previous bound holds. 
To bound the behavior of $g(\up)$ on the disk boundary near the lines $\pm \arg(\up) = 3\pi/4$ first observe that $g(\widebar{\up}) = \widebar{g(\up)}$ so it is sufficient to only consider $\arg(\up)$ near $3\pi/4$. 
Write $\up = R^-_\delta e^{i(3\pi/4 + \alpha)}$, the case when $|\up| = R^+_\delta$ can be treated identically. 
Letting $\K^- = \K^-(n,\delta)$, the first terms in the asymptotic expansion gives
\begin{multline*}
	g(\up) = e^{(R^-_\delta)^2 \sin 2\alpha} 
		\lb \cos \lb (2\pi \K^- +3\pi/4) \cos 2\alpha \rb 
		- i  \sin \lb (2\pi \K^- +3\pi/4) \cos 2\alpha \rb  \rb \\
	+  \frac{1}{2\sqrt{\pi } R^-_\delta} \lb \cos(3\pi/4 + \alpha) - i \sin (	3\pi/4 + \alpha) \rb + \bigo{(R_\delta^{-})^2}.
\end{multline*}
For $2|\alpha| \leq \arccos \lp \frac{ 2\pi K^-}{ 2\pi K^- + 3\pi/4} \rp$, the imaginary parts of the first two terms are both negative and so the sum has a larger (in absolute value) imaginary part than either term separately, hence 
\[
	|g(\up)| \geq | \Im g(\up) | \geq \frac{C}{R_\delta^-} 
	\quad \text{for} \quad 
	2|\alpha| \leq \arccos \lp \frac{ 2\pi K^-}{ 2\pi K^- + 3\pi/4} \rp.
\]
On the other hand for $\arccos \lp \frac{ 2\pi K^-}{ 2\pi K^- + 3\pi/4} \rp <  2|\alpha| \leq \pi/50$,
\[
	|e^{\up^2}| = e^{ (R_\delta^-)^2 \sin 2\alpha} 
	\gtrless e^{\pm 2\pi (\K^-+3/8) 
	\sqrt{1 - \lp \frac{ 2\pi K^-}{ 2\pi K^- + 3\pi/4} \rp^2 }} 
	= e^{\pm \sqrt{3}\pi \sqrt{\K^-} } \lb 1 + \bigo{\frac{1}{\sqrt{\K^-} } } \rb,
\]
where the upper (lower) inequality and signs are taken if $\alpha$ is positive (negative).
Now, since $\K^- \sim n \delta^2$, the exponential is already beyond all orders separated in scale from the algebraic terms in the expansion and we see that $|g(\up)| \leq C/R_\delta^-$ for $|\alpha| < \pi/100$. 
\end{proof}

%
Combining the results of Sections~\ref{sec:bulk zeros}-\ref{sec:saddle zeros} we can now prove Theorem~\ref{thm:zeros-outside}, the main result of Section~\ref{sec:taylor zeros}.
\begin{proof}[Proof of Theorem~\ref{thm:zeros-outside}]
Since $\mathcal{G} - \mathcal{G}_{1} = \frac{1}{n}\log{\frac{h}{h_{0}}}$, and $\frac{1}{n} \log{\frac{h}{h_{0}}}$ as well as its derivative are uniformly $\mathcal{O}\left( n^{-2} \right)$ in $\U$, there is a single level curve of $\Re \mathcal{G} = \frac{\log{n}}{2n}$ in $\U$, on which all zeros of $T_{2n-2}(\lambda z)$ in $\U$ must live.  As one traverses this level curve from $\partial \ball_{1,\delta}$ to the point where it intersects the vertical line $\{z: \Re z = \frac{1}{2 \lambda} \}$, $\Im \mathcal{G}$ is strictly monotone increasing.  Each root of $T_{2n-2}(\lambda z)$ in $\U$ must satisfy (\ref{zero condition 1}) for some integer $k$.  The monotonicity of both $\Im \mathcal{G}$ and $\Im \mathcal{G}_{1}$  along the associated level curves, and the fact that  
$\left| \Im \mathcal{G} - \Im \mathcal{G}_{1} \right| \le c n^{-2}$ uniformly in $\U$,
imply that the only roots of $T_{2n-2}(\lambda z)$ within $\U$ are those identified in Lemma \ref{lem:true zeros outside}.

We observe that: 1) the boundedness of the derivative $\partial_z \phase$ near the edge of the critical strip (and in any compact subset of $\U$) along with the smoothness of the level curve $\Re \mathcal{G}_1 = \tfrac{ \log n}{2n}$ implies that there exist a constant $c>0$ such that zeros $|z_{k,n} - z_{j,n}| > c/n$; and 2) Lemma~\ref{lem:true zeros outside} guarantees that $|z_{k,n} - \alpha_{k,n}| = \bigo{n^{-2}}$.

Now the left-most approximate zero in $\U$ corresponds to a root of $T_{2n-2}(\lambda z)$ that may or may not lie within $\U$.  Likewise, the approximate zero $\alpha_{\tilde{k},n}$ within the critical strip that is closest to the vertical line $\{ Re z = \frac{1}{2 \lambda} \}$ corresponds to a root of $T_{2n-2}(\lambda z)$ that may or may not lie within $\U$.  Similar considerations for those approximate roots near $\partial \ball_{1,\delta}$ show that, again, there could be up to $2$ additional roots (or 2 fewer roots) of $T_{2n-2}(\lambda z)$ near $\partial \ball_{1,\delta}$, because of boundary effects.  So we have shown that
\[
\left| \  \left| \left\{
 \mathcal{Z}_{n} \cap \U
\right\} \right| - |\mathcal{A}_n|  \  \right| \le 4.
\]

Combining the this observation with Lemmas~\ref{lem:outer zeros} and \ref{lem:true zeros outside} we get an expression for the number of true zeros $z_{k,n}$ outside the critical strip (by left-right symmetry we multiply $|\Acal_n|$ by $2$) which are bounded away from the stationary points $z =\pm 1$. Lemma~\ref{lem:local zeros} gives an exact count of the number of zeros in $\ball_{\pm1 , \delta}$ (of which half of each are in $\mathcal{F}_\lambda$). Summing these contributions we find that 
\[
	|\mathcal{Z}_n| = n 
		- \frac{\lambda \y}{2\pi} \log \lp \frac{ \lambda \y}{2\pi} \rp  
		+ \frac{\lambda \y}{2\pi } 
	 	- \frac{1}{4\pi \y} \log \lp \frac{ \lambda \y}{2\pi} \rp  
		- \frac{n \delta^2}{\pi} + 2\K^-(n,\delta) + \bigo{\log\!\log \lambda \y}.
\]
The result then follows from observing that 
\[
	\left| 2\K^-(n,\delta) - \frac{n \delta^2}{\pi} \right|
	= \left| 2\floor*{\frac{n\delta^2}{2\pi} - \frac{3}{8}}- \frac{n \delta^2}{\pi} \right|
	\le 2.
\]
\end{proof}

\section{Convergence rates of true zeros} 
\label{sec:true zeros}

\begin{table}
\centering
\begin{tabular}[t]{@{}l@{\qquad\quad}l@{\hskip 0.2\linewidth}l@{\qquad\quad}l@{}}
\toprule
$k$ & $|s_k - z_{k,n}|$ & $k$ & $|s_k - z_{k,n}|$  \\ \midrule
1 & $3.4293*10^{-34}$ & 7  & $1.5374*10^{-18}$  \\
2 & $6.9534*10^{-32}$ & 8  & $6.5531*10^{-17}$ \\
3 & $1.4748*10^{-30}$ & 9  & $3.6990*10^{-13}$  \\
4 & $9.8245*10^{-26}$ & 10 & $6.4702*10^{-12}$ \\
5 & $6.3374*10^{-24}$ & 11 & $5.2363*10^{-6}$  \\
6 & $7.1106*10^{-21}$ & & \\ \bottomrule
\end{tabular}
\caption{Tabulated here are the differences between the 11 numerical calculated of the zeros $z_{k,n}$ of $T_{202}(\lambda z)$ on the critical line depicted in Figure~\ref{fig:zeta}, and the first 11 (rescaled) zeros of the $\xi(s+1/2)$ function, denoted here as $s_k$.
\label{table:zeros}
} 
\end{table}

In this section we turn our attention to those zeros of $T_{2n-2}(\lambda z)$ which converge to roots of $\xi(z+1/2)$ as $n \to \infty$. Recall that Hurwitz theorem guarantees that near any root $s$ of order $m$ of $f(z)$ in the unscaled plane, there will be exactly $m$ zeros of $T_{n}(z)$ for all sufficiently large $n$, and that these will converge to $s$ as $n \to \infty$. 
These are the `Hurwitz zeros' of $T_{n}(z)$. 
In Figure~\ref{fig:zeta} there are 11 zeros of $T_{202}(\lambda z)$ on the positive critical line below the level curve $D_n^{(0)}$. The absolute error between these numerically computed Hurwitz zeros of the Taylor polynomial zeros and the first 11 nontrivial roots of the $\xi(\lambda z+1/2)$ function are given in Table~\ref{table:zeros}. The agreement is surprising good, particularly considering that the scaling factor $\lambda \approx 133 $ for $T_{202}(\lambda z)$. 

Suppose that $s$ is an order $m$ zero\footnote{It is widely believed, but unproven, that all zeros of $\xi$ are simple} of $f(s) = \xi(s+1/2)$, and suppose that $\lambda z = s+ \mu$ is a Hurwitz zero of the Taylor polynomial $T_{2n-2}(\lambda z)$. 
As $\Re s < \tfrac{1}{2} < \lambda$  our representation \eqref{represent} of $T_{2n}(\lambda z)$ gives
\begin{align}
	T_{2n-2}(s+\mu) = f(s+\mu) + 
	\frac{ (s+\mu)^{2n} f(\lambda) }{ \lambda^{2n} \sqrt{n}} h(s/\lambda+\mu/\lambda).
\end{align}
Expanding around $s$ gives
\begin{gather}
	T_{2n-2}(s+\mu) = \frac{f^{(m)}(s)}{m!} \mu^{m} + \frac{s^{2n} f(\lambda)}{\lambda^{2n} \sqrt{n}} h(s/\lambda) + R(s,\mu) 
\shortintertext{where $R$ is the explicit remainder}
	\begin{multlined}[.9\textwidth] \label{RRR}
	R(s,\mu) = \lb f(s+\mu) - \frac{f^{(m)}(s)}{m!} \mu^m \rb \\
	+ \frac{s^{2n} f(\lambda)}{\lambda^{2n} \sqrt{n} } \lb  
	h\lp \frac{s+\mu}{\lambda} \rp - h\lp \frac{s}{\lambda} \rp
	+ \lp (1+\frac{\mu}{s})^{2n} - 1 \rp h\lp \frac{s+\mu}{\lambda} \rp \rb.
	\end{multlined}
\end{gather}

Consider how the factor $\frac{s^{2n} f(\lambda)}{\lambda^{2n} \sqrt{n} }$ behaves as $n \to \infty$. Using Stirling's series for $\Gamma$,  we find
\begin{equation}
	f(\lambda) = \lp \frac{\lambda}{2\pi e} \rp^{\lambda/2} 
	\lp \frac{\lambda}{2\pi}\rp^{7/4}  2\sqrt{2} \pi 
	\lb 1 - \frac{1}{48\lambda} + \bigo{\lambda^{-2}} \rb
\end{equation}
where we note that $\zeta(\lambda+1/2) - 1$ is beyond all orders small and so makes no contribution to the asymptotic expansion. This can be further simplified using Lemma~\ref{lem:crit pts}; setting 
$u = \frac{2n}{\pi}$, and recalling the defining relation for $W=W(u)$, \ie, $W e^{W} = u$ we have
\begin{equation}
	\frac{\lambda}{2\pi} = \frac{u}{W} \lb 1 + \bigo{n^{-1}} \rb 
	= e^W \lb 1 + \bigo{n^{-1}} \rb
\end{equation}
which allows us to write
\begin{equation}
	f(\lambda) \sim \lp \frac{\lambda}{2\pi e} \rp^{\lambda/2} 
	\lp \frac{\lambda}{2\pi}\rp^{7/4}
	\sim  e^{2n(1-W^{-1} + \frac{7}{8n}W )  }. 
\end{equation}
so finally we have the asymptotic estimate
\begin{equation}\label{small factor}
	\left| \frac{ s^{2n} f(\lambda) }{ \lambda^{2n} \sqrt{n} } \right|
	\sim 
	\exp \lb -2n \lp \log \left| \frac{\lambda}{s} \right| - 1 + W^{-1} 
	- \frac{7}{8n} W + \frac{1}{4n} \log n \rp \rb
\end{equation}
which is uniformly exponentially small provided that $|s| < \frac{\lambda}{e}$.

\begin{thm}\label{true zero convergence}
	Let $s$ be a fixed zero of $\xi(s+1/2)$ of order $m$. Then there exist an $N_0 =N_0(s)$ such that for all $n\geq N_0$ the Taylor polynomial $T_{2n-2}(\lambda z)$ defined by \eqref{represent} and Lemma~\ref{lem:crit pts} has exactly $m$ zeros $\{ z^\mathrm{H}_{k,n} \}_{k=1}^m$ converging to $s/\lambda$. Moreover, these zeros converge at a super-exponential rate:
\begin{equation}
	\max_{1\leq k \leq m} \left| \lambda z^\mathrm{H}_{k,n} - s \right|  
	= \bigo{ \exp \lb -\frac{2n}{m} \lp \log \left| \frac{\lambda}{s} \right| - 1 + W^{-1} 
	- \frac{7}{8n} W + \frac{1}{4n} \log n \rp \rb }.
\end{equation}
\end{thm}

\begin{proof}
The first half of the theorem is just a restatement of Hurwitz's theorem in the case of Taylor polynomials. 
It remains to establish our superexponential bound on the rate of convergence. 
Let $s$ be a fixed root of order $m$ of $f(s)$, write $\lambda z = s + \mu$ and define the function 
\[
	g(\mu) := T_{2n-2}(s+\mu) - \frac{f^{(m)}(s)}{m!} \mu^{m}
	 = \frac{s^{2n} f(\lambda)}{\lambda^{2n} \sqrt{n}} h(s/\lambda) + R(s,\mu) .
\]
Let 
$
	\rho(n) = \left| \frac{ s^{2n} f(\lambda) }{ \lambda^{2n} \sqrt{n} } \right|^{1/m}.
$
It follows from \eqref{H} and \eqref{h0} that $h(z/\lambda)$ is analytic and bounded for all $z$ in the critical strip. Then using \eqref{RRR}, Taylor's remainder theorem and \eqref{small factor}, on the circle $|\mu| = A \rho(n)$ we have 
\[
	\begin{gathered}
	\left| \frac{f^{m}(s)}{m!} \mu^m \right| =  \frac{|f^{m}(s)| A^m }{m!} \rho(n)^m 
	\\
	|g(\mu)| \leq \rho(n)^m \| h \|_{L^\infty(\mathcal{S}_{1/(2\lambda)})} 
	+ \bigo{ n \rho(n)^{m+1} }
	\end{gathered}
\]
Taking $A$ and $N$ sufficiently large we use Rouche's theorem to conclude that $T_{2m-2}(s+\mu)$ has exactly $m$ zeros inside the circle or radius $A \rho(n)$. 
\end{proof}

\begin{rem}
Though we have only considered fixed roots $s$ in the above Theorem which do not scale with parameter $\lambda(n)$, these may also exhibit super-exponential convergence under certain assumptions. 
To consider growing roots, in the proof above one must include the asymptotic behavior of $s$ and $f^{(m)}(s)$. 
Essentially one must know that $s$ does not grow faster than $\lambda$, that as $|s|$ grows the order $m$ of the roots is bounded, and that the first non-zero derivative $f^{(m)}(s)$ is not too close to zero. 
Skipping the other details, the proof goes through as before where one considers the new radial scaling factor:
\[
	\tilde{\rho}(n) = \left| \frac{ s^{2n} f(\lambda) m! }{ f^{(m)}(s) \lambda^{2n} \sqrt{n} } \right|^{1/m}
\]
then in light of \eqref{small factor} the condition for super-exponential convergence amounts to knowing that the quantity
\[
	r_{n,s} = \log|\lambda| - \log|s| - \frac{1}{n} \log \frac{m!}{f^{(m)}(s)} \gg 1.
\]
and exponential convergence is maintained as long as $r_{n,s}$ is positive and bounded away from zero.
\end{rem}

\section{Extensions to a class of L-functions}

In this section we briefly explain how the results can be extended to a class of functions that includes many functions of interest in the theory of numbers, referred to as analytic L-functions.  

Suppose that a function $L$ is defined via a Dirichlet series,
\[
L(z) = \sum_{n=1}^{\infty}\frac{a_{n}}{n^{z}} \ , 
\]
with the $a_{n}$'s being real.  The following conditions are sufficient to extend the analysis described above to a collection of these analytic $L$ functions.
\begin{itemize}
\item[(A)] The function $L$ extends to an analytic function in $\mathbb{C}$.
\item[(B)] The function $L$ satisfies a functional equation of the form
\begin{align}
	\Lambda(z) := N^{z/2} \prod_{j=1}^{J} \Gamma_{\mathbb{R}}(z + \mu_{j})
	\prod_{k=1}^{K}  \Gamma_{\mathbb{C}}(z + \eta_{j}) L(z)  
	\quad = \quad \Lambda(1 - z) \ , 
\end{align}
where, following \cite{lmfdb},  $\Gamma_{\mathbb{R}}(s) = \pi^{-s/2} \Gamma \left( \frac{s}{2} \right)$ and $\Gamma_{\mathbb{C}}(s) = 2 ( 2 \pi)^{-s} \Gamma \left(s \right)$.
\item[(C)]  The function $L$ satisfies a polynomial bound in $\mbox{Im}(z)$ for $|z| \to \infty$:  for $y$ sufficiently large, $|L(x+iy) \le C |y|^{a}$ for some positive constants $C$ and $a$.
\item[(D)]  The analogue of the last two estimates in (\ref{zeta bounds}) for the function $L$ holds true.
\end{itemize}

\begin{rem}  The last two estimates in (\ref{zeta bounds}) for the function $\zeta$, or their analogues for $L$ (condition (D)), are more than what is needed for the asymptotic analysis of the Taylor approximants.  In order to establish the uniform asymptotic description of the Taylor approximants, it is sufficient to represent the Taylor approximant as an integral over two vertical lines, and then have enough analytical control on the phase function $\phi_{\lambda}$ in order to apply the steepest descent method.  This is guaranteed by conditions (A), (B), and (C) alone, and we present those results below.   The additional detailed control near the edge of the critical strip was used to confine the ``spurious zeros" of $T_{2n-2}(\xi; \lambda z)$, using our asymptotic analysis, establishing Theorem 4.1 and its Corollary.  One can easily state results analogous to Theorem 4.1 and its Corollary.  However, since this relies upon estimates which are known only in special cases (to the best of our knowledge), we will refrain from stating these conditional results.  Rather, in this section we will only state the extension of Theorem 3.1 to a general class of analytic L-functions whose existence is already established.
\end{rem}
\begin{rem}  Of course, we could also relax some of the above conditions, and there are in principle no additional obstacles if we permit complex Dirichlet coefficients $a_{n}$ (which can lead to a functional equation of the form $\Lambda(z) = \overline{\tilde{\Lambda}}(1-z)$, where $\tilde{\Lambda}$ is an L-function dual to $\Lambda$), nor if we permit the existence of a pole at $z=1$ (and hence at $z=0$).  But the additional complication is perhaps worth consideration only in specific examples.
\end{rem}
Under assumptions (A), (B), and (C), we can define the function $F(z) = \Lambda( 1/2 + z)$, and then express the rescaled Taylor polynomial of degree $2n -2$ for the function $F( z)$ via a contour integral over two vertical lines, as in (\ref{represent}):
\begin{gather}
\label{LfunctRep}
	T_{2n-2}(F; \lambda z) = F(\lambda z) 
	\left[ \chi(z) - \frac{e^{n \phi_{\lambda}(z)}}{\sqrt{n}}\mathcal{H}(z) \right], \\
		e^{n \phase(z)} := \frac {z^{2n} \mathcal{F}(\lambda) }{F(\lambda z) }, 
		\qquad
		\mathcal{H}(z) := \frac{ \sqrt{n}}{ 2\pi i } \int_{ \dscal_1 } 
		e^{-n \phase(s)} \frac{ds}{s-z}.
\end{gather}
Moreover, the phase function $\phi_{\lambda}$ can be expressed in manner analogous to (\ref{phase expand 0}):
\begin{multline}
	\phi_{\lambda}(z) 
	= 2 \log{z} + \frac{\lambda}{n} \left[ 
		\log{\left(  \frac{2^{K}}{N^{1/2}}\right)}  - 
		\left( \frac{J}{2} + K \right) \log{\frac{\lambda}{2 \pi}} 
		\right](z-1)  + \\
	- \frac{ \lambda}{n}  \left[ \frac{J}{2} + K \right] \left( 1 - z + z \log{z}  \right)
	- \frac{1}{n} \log{ L\left( \frac{1}{2} + \lambda z \right)} 
	+ \frac{1}{n} r(z;\lambda)
\end{multline}
where the function $r(z;\lambda)$, as well as its derivative in $z$, is bounded.  We may then compute $\phi_{\lambda}'$:
\begin{multline}
	\phi_{\lambda}'(z) = 
	\frac{2}{z} + \frac{\lambda}{n} \left[ 
	\log{\left(  \frac{2^{K}}{N^{1/2}}\right)}  
	- \left( \frac{J}{2} + K \right) \log{\frac{\lambda}{2\pi}} 
	\right]  + \\
	- \frac{ \lambda}{n}  \left[ \frac{J}{2} + K \right]  \log{z} 
	- \frac{\lambda}{n} \frac{L'}{L} 
	\left(  \frac{1}{2} + \lambda z \right) + \frac{1}{n} r'(z;\lambda)
\end{multline}

Following the arguments of Section 2, one may verify the analogue of Lemma \ref{lem:crit pts}, which sets the stage for the application of the steepest descent method. 

\begin{lemma}\label{lem:LfunctCritPts}
Suppose that the function $L$ satisfies assumptions (A), (B), and (C) above.  Then, for all sufficiently large $n$ there is a unique choice of $\lambda = \lambda(n)$, with $\lambda > 1/2$ (\ie right of the shifted critical strip) satisfying 
$
	\partial_z \phase(z) \Big|_{z=1} = 
	2 - (\lambda/n) \partial_\lambda \log F(\lambda) = 0.
$
This choice of $\lambda$ satisfies the relation
\begin{equation}\label{Lfunc lambda-func-asymp}
	2 + \frac{\lambda}{n}  \left[ \log{\left(  \frac{2^{K}}{N^{1/2}}\right)}  - \left( \frac{J}{2} + K \right) \log{\frac{\lambda}{2\pi}} \right]   
	= \bigo{\frac{1}{n}},
\end{equation}
and asymptotically 
\[
	\lambda = \lambda(n) 
	= \frac{4n}{J + 2 K} \left[  \W \lp  \frac{ 2n}{\pi ( J + 2 K)} \lp \frac{N}{4^{K}} \rp^{\frac{1}{J+2K}}\rp \right]^{-1} \lb 1 + \bigo{n^{-1}} \rb.  	
\]
Here $\W(z)$ is the branch of the inverse function to $\W e^{\W} = z$ which is real and increasing for $z \in (-e^{-1},\infty)$ sometimes called the Lambert-$\W$ function. 

Moreover, for this choice of $\lambda$, the critical point at $z=1$ is simple and
\begin{equation}\label{Lfunc phase constant}
	\phase''(1) = -2 + \bigo{ \frac{1}{\log n} }. 
\end{equation}
In addition, for any $\sigma > 1$, there exists $N(\sigma)$ sufficiently large so that for all $n>N(\sigma)$, the only critical point of $\phi_{\lambda}$ in $\left\{z: \mbox{Re}(z) > \frac{\sigma}{2 \lambda}\right\}$ is $z=1$.
\end{lemma}

The analog of Theorem 3.1, establishing the uniform asymptotic behavior of the Taylor approximants $T_{2n-2}(F; \lambda z)$, is also clear, under these assumptions.  The computations are a bit more involved because of the flurry of Gamma functions, but are otherwise straightforward.

The analogue of the analytic transformation $w(z)$, defined in the beginning of Section 3, is defined via
\begin{equation}\label{LfunctWRelation}
	w^2 = \phase(z) = \frac{\phase''(1)}{2}  (z-1)^2 \lb 1 + \bigo{z-1}  \rb .
\end{equation}
This function obeys all of the properties described in the beginning of Section 3.  In particular, when restricted to any sufficiently small neighborhood $\ball_{1,\delta}$ of $z=1$ (or $\ball_{-1,\delta}$ of $z=-1$), it is an invertible conformal map onto a bounded neighborhood of $w = 0$, and the branch is chsen so that $w$ maps the vertical line $\dscal_1$ locally to a nearly horizontal contour in the $w$-plane oriented left-to-right.  In addition, we may use symmetry so that $w(z) = w(-z)$ for $z \in \ball_{-1,\delta}$.  Moreover, the estimate on $\phase''(1)$ in Lemma~\ref{lem:LfunctCritPts} implies that $w=w(z)$ is nearly isometric for $z$ near 1 and $n \gg 1$. We again fix the neighborhoods $\ball_{\pm1, \delta}$ by requiring that $B_{\pm 1, \delta}$ are, for any sufficiently small $\delta>0$, the two pre-images of the disk of radius $\delta$ in the $w$-plane:
\begin{equation}\label{LfunctBallDef}
	w \lp \ball_{\pm 1, \delta} \rp = \{ w \in \C \,:\, |w| < \delta \}
\end{equation}
and we let $\ball_\delta = \ball_{1,\delta} \cup \ball_{-1,\delta}$.

\begin{thm}\label{thm:LFuncT asymp}
Suppose the function $L$ satisfies assumptions (A), (B), and (C) above.  Let $\lambda = \lambda(n)$ be as described in Lemma~\ref{lem:LfunctCritPts}, $\chi(z)$ the characteristic function of the set $| \Re z | < 1$, and $\mathcal{H}_0(z)$, defined via
\begin{equation}\label{Lfuncth0}
	\mathcal{H}_0(z) = \frac{1}{ \sqrt{2\pi | \phase''(1) |} } \frac{2}{1-z^2}.
\end{equation}
be the leading order stationary phase approximation of $\mathcal{H}(z)$.  Then as $n \to \infty$ the Taylor polynomials described by \eqref{LfunctRep} admit the asymptotic expansion
\begin{align}\label{Lfunc T asymp}
	T_{2n-2}(F; \lambda z) = 
	T_{2n-1}(F; \lambda z) =
	\begin{dcases}
		F( \lambda z) \lb 
		\chi(z) 
		- \frac{ e^{n \phase(z)} }{ \sqrt{n} } \lp \mathcal{H}_0(z) 
		+ \Ecal(z) \rp
		\rb
		& z \in \C \backslash \ball_\delta 
		\\ \bigskip
		F( \lambda z) 
		\lb 
		\frac{1}{2} \erfc(i \sqrt{n} w(z) ) 
		- \frac{e^{n\phase}(z) }{ \sqrt{n} } \Ecal(z)
		\rb
		& z \in  \ball_\delta.
	\end{dcases}
\end{align}
where the residual error function $\Ecal(z)$ is bounded, analytic in $\C \backslash \lp (\dscal_1 \backslash \ball_\delta) \cup \partial \ball_\delta \rp$, and satisfies
\begin{equation}
	\Ecal(z) = \begin{dcases}
		\bigo{n^{-1}} 
			& z \in \ball_\delta^c \\
		h_0(z) + \frac{1}{2i \sqrt{\pi} w(z)} + \bigo{n^{-1}}
			& z \in \ball_\delta.
	\end{dcases}
\end{equation}
\end{thm}

\section*{Acknowledgements}
R. Jenkins was supported by the National Science Foundation under grant DMS-1418772.
K. D. T-R McLaughlin was supported by the National Science Foundation under grant DMS-1401268.  The authors thank Nigel Pitt for useful conversations.

\printbibliography

\end{document}